\documentclass[reqno,12pt]{amsart}
 \usepackage{amsmath}
 \usepackage{amsthm}
 \usepackage{amssymb}
 \usepackage{latexsym,longtable}
 \usepackage{graphicx}
 \usepackage{multicol}
 \usepackage{mathrsfs}
 \usepackage{pstricks, pst-node}
 \usepackage{young}
 \usepackage[vcentermath,enableskew,stdtext]{youngtab}
 \usepackage[left=1.1in,right=1.1in,top=1.09in,bottom=1.13in]{geometry}
 \usepackage[all]{xy}
    \SelectTips{cm}{10}     %use the nicer arrowheads
    \everyxy={<2.5em,0em>:} % Sets the scale I like
 \usepackage{fancyhdr}      %%%%%%%%%%%% Pagestyle stuff %%%%%%%%%%%%%%%
 \usepackage{multicol}
 \usepackage{multirow}
 \usepackage{enumitem}
 \usepackage{stmaryrd}
 \usepackage{etex}
 \usepackage{tikz}
 \usepackage{float}
 \usepackage{array}
 \restylefloat{figure}
\usepackage{amsfonts,epsfig,color}
\newcounter{ctr}
\theoremstyle{plain}
\newtheorem{theorem}{Theorem}[section]
\newtheorem{lemma}[theorem]{Lemma}
\newtheorem{corollary}[theorem]{Corollary}
\newtheorem{proposition}[theorem]{Proposition}

\newtheorem{propdef}[theorem]{Proposition-Definition}

\theoremstyle{definition}
\newtheorem{definition}[theorem]{Definition}

\newtheorem{remark}[theorem]{Remark}
\newtheorem{example}[theorem]{Example}

\newcommand{\ignore}[1]{}

\newcommand{\CC}{\ensuremath{\mathbb{C}}}

\newcommand{\End}{\text{\rm End}}

\renewcommand{\H}{\ensuremath{\mathscr{H}}}

\renewcommand{\hom}{\text{\rm Hom}}

\newcommand{\ind}{\text{\rm Ind}}

\renewcommand{\L}{\ensuremath{\mathscr{L}}}

\newcommand{\QQ}{\ensuremath{\mathbb{Q}}}

\newcommand{\ZZ}{\ensuremath{\mathbb{Z}}}
\newcommand{\be}{\begin{equation}}
\newcommand{\ee}{\end{equation}}
\newcommand{\Res}{\text{\rm Res}}
\renewcommand{\S}{\ensuremath{\mathcal{S}}}
\newcommand{\tsr}{\ensuremath{\otimes}}
\newcommand{\tsrvw}{\tsr}

\newcommand{\C}{\ensuremath{C^{\prime}}} %canonical basis

\newcommand{\liftCp}{\ensuremath{\tilde{C}^{\prime}}}
\newcommand{\liftCup}{\ensuremath{\tilde{C}}}

\newcommand{\br}[1]{\ensuremath{\overline{#1}}}
\newcommand{\nsbr}[1]{{\ensuremath{\check{#1}}}}

\renewcommand{\u}{\ensuremath{u}}  %q^{1/2}
\newcommand{\ui}{\ensuremath{u^{-1}}} %q^{-1/2}

 \newcommand{\dual}[1]{\ensuremath{{#1}^\vee}}
\newcommand{\gd}{\ensuremath{\trianglerighteq}} %dominance order
\newcommand{\gdneq}{\ensuremath{\triangleright}}

\newcommand{\ldneq}{\ensuremath{\triangleleft}}

\newcommand{\tleftright}{\ensuremath{\leftrightsquigarrow}}
\newcommand{\dkt}[1]{\raisebox{1pt}{\ensuremath{{\underset{\tiny#1}{\tleftright}}}}}

\newcommand{\sh}{\text{\rm sh}}

\newcommand{\transpose}[1]{{#1}^T}

\newcommand{\ngb}{\ensuremath{\tiny \hspace{-.36mm}-\hspace{-.7mm}}}

\newcommand{\nsH}{\nsbr{\H}}
\newcommand{\nsP}{\nsbr{\mathscr{P}}}

\DeclareMathOperator{\trace}{tr}
\newcommand{\two}{[2]}
\newcommand{\sP}{\mathcal{P}}
\newcommand{\sQ}{\mathcal{Q}}
\newcommand{\field}{\ensuremath{K}}
\newcommand{\klo}[1]{\ensuremath{\leq_{#1}}}
\newcommand{\klocov}[1]{\ensuremath{\preceq_{#1}}}

\newcommand{\halfr}{\lfloor\frac{r}{2}\rfloor}

\newcommand{\Wedge}{\ensuremath{\Lambda}}

\newcommand{\nssym}[2]{\ensuremath{\nsbr{S}^{#1}\nsbr{#2}}}
\newcommand{\nswedge}[2]{\ensuremath{\nsbr{\Wedge}^{#1}\nsbr{#2}}}

\newcommand{\f}{\ensuremath{[2]^2}}
\newcommand{\proj}{\ensuremath{\nsbr{p}^0}}
\newcommand{\projres}{\ensuremath{\nsbr{p}^1}}

        \newlength{\cellsizeCol}
        \newcommand\column[1]{
        \vtop{
        \let\\=\cr
        \baselineskip=-16000pt
        \lineskiplimit=16000pt
        \lineskip=0pt
        \halign{& \columncell{##} \cr \noalign{\hrule height \arrayrulewidth width \cellsizeCol} #1 \cr \noalign{\hrule height \arrayrulewidth width \cellsizeCol} \crcr} }
        \hspace{-.73ex}}

\newcommand\columnL[1]{
        \vtop{
        \let\\=\cr
        \baselineskip=-16000pt
        \lineskiplimit=16000pt
        \lineskip=0pt
        \halign{& \columncellFull{##} \cr \noalign{\hrule height \arrayrulewidth width \cellsizeCol} #1 \cr \noalign{\hrule height \arrayrulewidth width \cellsizeCol} \crcr} }
        \hspace{-.73ex}}

\newcommand\columnBB[1]{
        \vtop{
        \let\\=\cr
        \baselineskip=-16000pt
        \lineskiplimit=16000pt
        \lineskip=0pt
        \halign{& \columncellFullB{##} \cr \noalign{\hrule height \arrayrulewidth width \cellsizeCol} #1 \cr \noalign{\hrule height \arrayrulewidth width \cellsizeCol} \crcr} }
        \hspace{-.73ex}}
\newcommand\columnB[1]{
        \vtop{
        \let\\=\cr
        \baselineskip=-16000pt
        \lineskiplimit=16000pt
        \lineskip=0pt
        \halign{& \columncellFullBB{##} \cr \noalign{\hrule height .3mm width \cellsizeCol} #1 \cr \noalign{\hrule height .3mm width \cellsizeCol} \crcr} }
        \hspace{-.73ex}}

\newlength{\oldSize}

\newcommand\mybox[1]{
\vcenter{
\let\\=\cr
\baselineskip=-16000pt \lineskiplimit=16000pt \lineskip=0pt
\halign{&\boxcell{##}\cr\vline#1\vline\crcr}}}

\newcommand{\boxcell}[1]{{%
\unitlength=\cellsizeCol
\begin{picture}(1,1)
\put(0,0){\makebox(1,1){$#1$}}
\put(0,0){\line(1,0){1}}
\put(0,1){\line(1,0){1}}
\end{picture}%
}}

\newcommand\pad[1]{
\vtop{
\let\\=\cr
\baselineskip=-16000pt
\lineskiplimit=16000pt
\lineskip=0pt
\halign{& \inviscell{##} \cr #1 \crcr} }
\hspace{-.73ex}}

\newcommand{\inviscell}[1]{{%
\unitlength=\cellsizeCol
\begin{picture}(1,1)
\put(0,0){\makebox(1,1){$#1$}}
\end{picture}%
}}

\newcommand{\columncell}[1]{{%
\unitlength=\cellsizeCol
\begin{picture}(1,1)
\put(0,0){\makebox(1,1){\footnotesize \centering $#1$}}
\put(0,0){\line(0,1){1}}
\put(1,0){\line(0,1){1}}
\end{picture}%
}}

\newcommand{\columncellFull}[1]{{%
\unitlength=\cellsizeCol
\begin{picture}(1,1)
\put(0,0){\makebox(1,1){\centering $#1$}}
\put(0,0){\line(0,1){1}}
\put(1,0){\line(0,1){1}}
\end{picture}%
}}

\newcommand{\columncellFullB}[1]{{%
\unitlength=\cellsizeCol
\begin{picture}(1,1)
\put(0,0){\makebox(1,1){$\mathbf{#1}$}}
\put(0,0){\line(0,1){1}}
\put(1,0){\line(0,1){1}}
\end{picture}%
}}

\newcommand{\columncellFullBB}[1]{{%
\unitlength=\cellsizeCol
\begin{picture}(1,1)
\linethickness{0.3mm}
\put(0,0){\makebox(1,1){$\mathbf{#1}$}}
\put(0,0){\line(0,1){1}}
\put(1,0){\line(0,1){1}}
\end{picture}%
}}

\newlength{\cellsize}
\newcommand\tableau[1]{
\vcenter{
\let\\=\cr
\baselineskip=-16000pt \lineskiplimit=16000pt \lineskip=0pt
\halign{&\tableaucell{##}\cr#1\crcr}}}

\newcommand{\tableaucell}[1]{{%
\def \arg{#1}\def \void{}%
\ifx \void \arg
\vbox to \cellsize{\vfil \hrule width \cellsize height 0pt}%
\else \unitlength=\cellsize
\begin{picture}(1,1)
\put(0,0){\makebox(1,1){$#1$}}
\put(0,0){\line(1,0){1}}
\put(0,1){\line(1,0){1}}
\put(0,0){\line(0,1){1}}
\put(1,0){\line(0,1){1}}
\end{picture}%
\fi}}

\newlength{\colskip}
\newlength{\dwidth}

\newcommand{\dualtheta}{\ensuremath{\#}}
\newcommand{\dualone}{\ensuremath{\diamond}}

\begin{document}

\address{University of Michigan, Department of Mathematics, 2074 East Hall, 530 Church Street, Ann Arbor, MI 48109-1043}
\email{jblasiak@gmail.com} 

\author{Jonah Blasiak}\thanks{The author was partially supported by an NSF postdoctoral fellowship.}
\title{Representation theory of the nonstandard Hecke algebra}
\maketitle
%??? should we mention that we really study \field \nsH -irreducibles?
\begin{abstract}
The nonstandard Hecke algebra $\nsH_r$ was defined by Mulmuley and Sohoni to study the Kronecker problem.  We study a quotient $\nsH_{r,2}$ of  $\nsH_r$, called the nonstandard Temperley-Lieb algebra,  which is a subalgebra of the symmetric square of the Temperley-Lieb algebra $\text{TL}_r$. We give a complete description of its irreducible representations.
We find that the restriction of an  $\nsH_{r,2}$-irreducible to  $\nsH_{r-1,2}$ is multiplicity-free, and as a consequence, any $\nsH_{r,2}$-irreducible has a seminormal basis that is unique up to a diagonal transformation.
\end{abstract}

%The nonstandard Hecke algebra \check{\mathscr{H}}_r was defined by Mulmuley and Sohoni to study the Kronecker problem.  We study a quotient \check{\mathscr{H}}_{r,2} of  \check{\mathscr{H}}_r, called the nonstandard Temperley-Lieb algebra,  which is a subalgebra of the symmetric square of the Temperley-Lieb algebra TL_r. We give a complete description of its irreducible representations.
%We find that the restriction of an  \check{\mathscr{H}}_{r,2}-irreducible to  \check{\mathscr{H}}_{r-1,2} is multiplicity-free, and as a consequence, any \check{\mathscr{H}}_{r,2}-irreducible has a seminormal basis that is unique up to a diagonal transformation.

\section{Introduction}
%In \cite{GCT4}, Mulmuley and Sohoni attempt to use canonical bases to understand Kronecker coefficients by constructing an algebra defined over $\ZZ[\u,\ui]$ that carries some of the information of the Hopf algebra structure of the group algebra $\ZZ \S_r$ and specializes to it at $\u=1$.  Specifically, they construct the nonstandard Hecke algebra  $\nsH_r$ (denoted  $B_r$ in \cite{GCT4}), which is a subalgebra of the tensor square of the Hecke algebra $\H_r$ such that the inclusion $\bar{\Delta} : \nsH_r \hookrightarrow \H_r \tsr \H_r$ is  $\u$-analogue of the coproduct  $\Delta$ of $\ZZ \S_r$ (see Definition \ref{d nonstandard Hecke algebra}).  The goal is then to break up the Kronecker problem into two steps \cite{GCT7}:

Let $\H_r$ be the type  $A_{r-1}$ Hecke algebra over $\mathbf{A} = \ZZ[\u, \ui]$ and  set  $\field := \QQ(\u)$.
%(see  \textsection\ref{ss the Hecke algebra H(W)} for details).
The nonstandard Hecke algebra $\nsH_r$ is the subalgebra of  $\H_r \tsr \H_r$ generated by
\[\sP_i := \C_{s_i} \tsrvw \C_{s_i} + C_{s_i} \tsrvw C_{s_i}, \ i \in [r-1],\]
where $\C_{s_i}$ and $C_{s_i}$ are the simplest lower and upper Kazhdan-Lusztig basis elements, which are proportional to the trivial and sign idempotents of the parabolic sub-Hecke algebra $\field (\H_r)_{\{s_i\}}$.
The nonstandard Hecke algebra  was introduced by Mulmuley and Sohoni in \cite{GCT4} to study the Kronecker problem. The hope was that the  inclusion  $ \nsbr{\Delta} : \nsH_r \to \H_r \tsr \H_r$ would quantize the coproduct $\Delta: \ZZ\S_r \to \ZZ\S_r \tsr \ZZ \S_r$ of the group algebra $\ZZ \S_r$ and canonical basis theory could be applied to obtain formulas for Kronecker coefficients.  Unfortunately, this does not work in a straightforward way since the algebra $\nsH_r$ is almost as big as $\H_r \tsr \H_r$ and has  $\mathbf{A}$-rank much larger than $r!$, even though  $\nsbr{\Delta}$ is in a certain sense the quantization of $\Delta$ with image as small as possible (see \cite[Remark 11.4]{BMSGCT4}). Nonetheless, as discussed in \cite[\textsection1]{BMSGCT4}, \cite{GCT7, canonical}, and briefly in this paper, the nonstandard Hecke algebra may still be useful for the Kronecker problem.

  %some of its representation theory is discussed in \cite[\textsection11]{BMSGCT4} and \cite{GCT7}, including a complete description  $\field \nsH_{3}$ and  $\field \nsH_{4}$-irreducibles; the problem of constructing a canonical basis for  $\nsH_r$ is discussed and \cite[\textsection19]{BMSGCT4} and \cite{canonical}.  The main purpose of this paper is to determine the irreducibles of the nonstandard Temperley-Lieb algebra $ \field \nsH_{r,2}$, which is a quotient algebra of  $\field \nsH_{r}$.  Here we assemble some basic facts about $\nsH_r$ from \cite{Bnsbraid, GCT4, BMSGCT4} and prove a few new ones. ???introduction
Though the nonstandard Hecke algebra has yet to prove its importance for the Kronecker problem, it is an interesting problem in its own right to determine all the irreducible representations of $\field \nsH_{r}$. This problem is difficult, but within reach. In this paper, we solve an easier version of this problem.

It is shown in \cite{BMSGCT4} that  $\field \nsH_r$ is semisimple.
Let $\tau$ be the flip involution of $\H_r\tsrvw\H_r$ given by $h_1\tsrvw h_2\mapsto h_2 \tsrvw h_1$ and let $\theta : \H_r \to \H_r$ be the $\mathbf{A}$-algebra involution defined by $\theta(T_{s_i}) = - T_{s_i}^{-1},\ i \in [r-1]$.  Twisting an  $\H_r$-irreducible by $\theta$ corresponds to transposing its shape.
The algebra $\nsH_r$ is a subalgebra of $(S^2 \H_r)^{\theta \tsrvw \theta}$, the subalgebra of  $\H_r \tsrvw \H_r$ fixed by $\theta \tsrvw \theta$ and $\tau$.
Based on computations for $r \leq 6$, it appears that most of the $\field \nsH_r$-irreducibles are restrictions of $\field (S^2 \H_r)^{\theta \tsrvw \theta}$-irreducibles, except for the trivial and sign representations of  $\field \nsH_r$.
%might be nice to have this in the introduction, but it is weird that we don't talk about it anywhere else

In this paper we focus on the simpler problem of determining the irreducibles of the nonstandard Temperley-Lieb algebra $\nsH_{r,2}$, which is a quotient of  $\nsH_{r}$. The algebra $\nsH_{r,2}$ is the subalgebra of $\H_{r,2} \tsr \H_{r,2}$ generated by $\sP_i := \C_{s_i} \tsrvw \C_{s_i} + C_{s_i} \tsrvw C_{s_i}, \ i \in [r-1]$, where $\H_{r,2}$ is the Temperley-Lieb algebra (see \textsection\ref{s irreducibles of nshr2}).

The main result of this paper is a complete description of the $\field \nsH_{r,2}$-irreducibles (Theorem \ref{t nsH irreducibles two row case}).  There are no surprises here: it is fairly easy to show that $\field (\H_{r,2} \tsr \H_{r,2})$-irreducibles decompose into certain $\field \nsH_{r,2}$-modules.  The difficulty is showing that these modules are actually irreducible. We prove this by induction on  $r$ and by computing the action of  $\sP_{r-1}$ on these modules in terms of canonical bases.
To carry out these computations, we use results from \cite{BProjected} about projecting the upper and lower canonical bases of a  $\field \H_{r}$-irreducible $M_\lambda$ onto its $\field \H_{r-1}$-irreducible isotypic components. We also use the well-known fact that the edge weight $\mu(x,w)$, $x, w \in \S_r$, of the  $\S_r$-graph  $\Gamma_{\S_r}$ is equal to 1 whenever $x$ and $w$ differ by a dual Knuth transformation (see  \textsection\ref{ss upper and lower canonical basis of H(W)} and  \textsection\ref{ss dual equivalence}).
%Once these results are in place, the only thing we need to know about the canonical bases is that the edge weight $\mu(x,w)$ is equal to 1 whenever $x$ and $w$ differ by a dual Knuth transformation (see  \textsection\ref{ss dual equivalence}).

One consequence of Theorem \ref{t nsH irreducibles two row case} is that the restriction
of a $ \field \nsH_{r,2}$-irreducible to $\field \nsH_{r-1,2}$ is multiplicity-free.
Thus each $\field \nsH_{r,2}$-irreducible  has a seminormal basis (in the sense of \cite{RamSeminormal}---see Definition \ref{d seminormal}) that is unique up to a diagonal transformation.
This can also be used to define a seminormal basis for any  $\field(\H_{r,2}\tsr\H_{r,2})$-irreducible.
%Though nothing was particularly surprising up to this point, this is now quite interesting.
Even though the irreducibles of  $\field \nsH_{r,2}$ are close to those of  $\field(\H_{r,2}\tsr\H_{r,2})$, the nonstandard Temperley-Lieb algebra offers something new:
the seminormal basis  of $M_\lambda\tsr M_\mu$ using the chain  $\field\nsH_{J_1} \subseteq \cdots \subseteq \field \nsH_{J_{r-1}} \subseteq \field\nsH_{J_r}$ is significantly different from the seminormal basis using the chain $\field(\H_{1,2} \tsr \H_{1,2}) \subseteq \cdots \subseteq \field (\H_{r-1,2} \tsr \H_{r-1,2}) \subseteq \field (\H_{r,2} \tsr \H_{r,2})$, where $\nsH_{J_k}$ is the subalgebra of $\nsH_{r,2}$ generated by $\sP_1, \dots, \sP_{k-1}$.

We are interested in these seminormal bases primarily as a tool for constructing a canonical basis of a $\field \nsH_{r,2}$-irreducible that is compatible with its decomposition into irreducibles at $\u =1$, as described in \cite[\textsection19]{BMSGCT4}.
Thus even though the representation theory of the nonstandard Hecke algebra alone is not enough to understand Kronecker coefficients,
there is hope that the seminormal bases will yield a better understanding of Kronecker coefficients.
In fact, \cite{canonical} gives a conjectural scheme for constructing a canonical basis of a $\field \nsH_{r,2}$-irreducible using its seminormal basis,
but this remains conjectural and we do not know how to use it to understand Kronecker coefficients.

This paper is organized as follows:
sections \ref{ss type A combinatorics preliminaries}--\ref{s the nonstandard Hecke algebra} are preparatory:  \textsection\ref{s canonical bases of hecke algebras} reviews the necessary facts about canonical bases of $\H_r$ and their behavior under projection onto $ \field \nsH_{r-1}$-irreducibles; \textsection\ref{s the nonstandard Hecke algebra} gives some basic results about the representation theory of $\nsH_r$.
Section \ref{s irreducibles of nshr2} contains the statement and proof of the main theorem.  Then in  \textsection\ref{s Seminormal bases}, seminormal bases of $\field \nsH_{r,2}$-irreducibles are defined, and in  \textsection\ref{s enumerative consequence}, the dimension of $\field \nsH_{r,2}$ is determined.

\section{Partitions and tableaux}
\label{ss type A combinatorics preliminaries}
A \emph{partition} $\lambda$ of  $r$ of length $\ell(\lambda) = l$ is a sequence $(\lambda_1, \ldots, \lambda_l)$ such that  $\lambda_1 \geq \cdots \geq \lambda_l > 0$ and $r = \sum_{i=1}^l \lambda_i$. The notation $\lambda \vdash r$ means that $\lambda$ is a partition of  $r$.
Let $\mathscr{P}_r$ denote the set of partitions of size $r$ and  $\mathscr{P}'_r$ the subset of $\mathscr{P}_r$ consisting of those partitions that are not a single row or column shape.  The symbols $\gd, \gdneq$ will denote dominance order on partitions.
The conjugate partition $\lambda'$ of a partition $\lambda$ is the partition whose diagram is the transpose of the diagram of $\lambda$.

The set of standard Young tableaux is denoted SYT and the subset of SYT of shape $\lambda$ is denoted SYT$(\lambda)$.
Tableaux are drawn in English notation, so that entries increase from north to south along columns and increase from west to east along rows. For a tableau $T$, $\sh(T)$ denotes the shape of $T$.

For a word $\mathbf{k} = k_1 k_2\dots k_r$,  $k_i \in \ZZ_{> 0}$, let $P(\mathbf{k}), Q(\mathbf{k})$ denote the insertion and recording tableaux produced by the Robinson-Schensted-Knuth (RSK) algorithm applied to  $\mathbf{k}$.
%We abbreviate $\sh(P(\mathbf{k}))$ simply by $\sh(\mathbf{k})$.
%Let $Z_\lambda$ be the superstandard tableau of shape and content $\lambda$---the tableau whose $i$-th row is filled with $i$'s.
The notation $\transpose{Q}$ denotes the transpose of an SYT $Q$, so that $\sh(\transpose{Q}) = \sh(Q)'$.

Let $T$ be a tableau of shape $\lambda$. If $b$ is a square of the diagram of $\lambda$, then $T_b$ denotes the entry of $T$ in the square $b$. If $\nu \subseteq \lambda$, then $T_\nu$ denotes the subtableau of $T$ obtained by restricting $T$ to the diagram of $\nu$.

Let $\lambda$ and $\mu$ be partitions of $r$. Throughout this paper, $a_1, \dots, a_{k_\lambda}$ (resp. $b_1, \dots, b_{k_\mu}$) will denote the outer corners of the diagram of $\lambda$ (resp. $\mu$) labeled so that $a_{i+1}$ lies to the east of $a_i$ (resp. $b_{i+1}$ lies to the east of $b_i$), as in the following example.
\setlength{\cellsize}{2ex}
\be \label{e ai definition}
\begin{array}{c@{\hspace{1in}}c}
{\tiny \tableau{\ & \ & \ & \ & \ & \ & \ & \ & a_4 \\
 \ & \ & \ & \ & \ & \ & a_3 \\
 \ & \ & \ & \ \\
 \ & \ & \ & a_2 \\
 \ & a_1}}
 &
 {\tiny \tableau{\ & \ & \ & \ & \ & \ & \ & \ \\
 \ & \ & \ & \ & \ & \ & \ & b_2 \\
 \ & \ & \ & \ & \ \\
 \ & \ & \ & \ & b_1}} \\
 \lambda & \mu
\end{array}
\ee

\section{Canonical bases of the Hecke algebra $\H_r$}\label{s canonical bases of hecke algebras}
Here we recall the definition of the Kazhdan-Lusztig basis elements $C_w$ and  $\C_w$ and review the connection between cells in type $A$ and tableaux combinatorics,
following \cite{BProjected}.
We then discuss dual equivalence graphs and recall some results of \cite{BProjected} about projecting canonical bases, which will make these bases fairly easy to work with in the proof of Theorem \ref{t nsH irreducibles two row case}.

We work over the ground rings $\mathbf{A} = \ZZ[\u, \ui]$ and $\field = \QQ(\u)$. Define $\field_0$ (resp. $\field_\infty$) to be the subring of $\field$ consisting of rational functions with no pole at $\u = 0$ (resp. $\u = \infty$).

Let $\br{\cdot}$ be the involution of $\field$ determined by $\br{u} = \ui$; it restricts to an involution of $\mathbf{A}$.
For a nonnegative integer $k$, the $\br{\cdot}$-invariant quantum integer is $[k] := \frac{\u^k - \u^{-k}}{\u - \ui} \in \mathbf{A}$.
We also use the notation $[k]$ to denote the set $\{1,\ldots,k\}$, but these usages should be easy to distinguish from context.

\subsection{The Hecke algebra $\H(W)$}
\label{ss the Hecke algebra H(W)}
Let $(W, S)$ be a Coxeter group with length function $\ell$ and Bruhat order $<$. If $\ell(vw)=\ell(v)+\ell(w)$, then $vw = v\cdot w$ is a \emph{reduced factorization}. The \emph{right descent set} of $w \in W$ is $R(w) = \{s\in S : ws < w\}$.

For any $L\subseteq S$, the \emph{parabolic subgroup} $W_L$ is the subgroup of $W$ generated by $L$.

The \emph{Hecke algebra} $\H(W)$ of $(W, S)$ is the free $\mathbf{A}$-module with standard basis $\{T_w :\ w\in W\}$ and relations generated by
\be \label{e Hecke algebra def} \begin{array}{ll}T_vT_w = T_{vw} & \text{if } vw = v\cdot w\ \text{is a reduced factorization},\\
(T_s - \u)(T_s + \ui) = 0 & \text{if } s\in S.\end{array}\ee
%We remark that the $q$ here is frequently $q^{1/2}$ in the literature on Hecke algebras, as it is, for instance, in \cite{KL}.

%For each $L\subseteq S$, $\H(W)_L$ denotes the subalgebra of $\H(W)$ with $\mathbf{A}$-basis $\{T_w:\ w\in W_L\}$, which is isomorphic to $\H(W_L)$.

%Pasted in from canbasliftsFeb11
\subsection{The upper and lower canonical basis of $\H(W)$}
\label{ss upper and lower canonical basis of H(W)}
The \emph{bar-involution}, $\br{\cdot}$, of $\H(W)$ is the additive map from $\H(W)$ to itself extending the $\br{\cdot}$-involution of $\mathbf{A}$ and satisfying $\br{T_w} = T_{w^{-1}}^{-1}$. Observe that $\br{T_{s}} = T_s^{-1} = T_s + \ui - u$ for $s \in S$. Some simple $\br{\cdot}$-invariant elements of $\H(W)$ are $\C_\text{id} := T_\text{id}$, $C_s := T_s - \u = T_s^{-1} - \ui$, and $\C_s := T_s + \ui = T_s^{-1} + u$, $s\in S$.
%Also note that $\C_s = C_s + [2]$, $s \in S$.

Define the lattices $(\H_r)_{\ZZ[\u]} := \ZZ[\u] \{ T_w : w \in W \}$ and $(\H_r)_{\ZZ[\ui]} := \ZZ[\ui] \{ T_w : w \in W \}$ of $\H_r$.
It is shown in \cite{KL} that
\be
\parbox{13cm}{for each $w \in W$, there is  a unique element $C_w \in \H(W)$ such that $\br{C_w} = C_w$ and $C_w$ is congruent to $T_w \mod \u (\H_r)_{\ZZ[\u]}$.}
\ee
The $\mathbf{A}$-basis $\Gamma_W := \{C_w: w\in W\}$
is the \emph{upper canonical basis} of $\H(W)$ (we use this language to be consistent with that for crystal bases).
Similarly,
\be
\parbox{13cm}{for each $w \in W$, there is  a unique element $\C_w \in \H(W)$ such that $\br{\C_w} = \C_w$ and $\C_w$ is congruent to $T_w \mod \ui (\H_r)_{\ZZ[\ui]}$.}
\ee
The $\mathbf{A}$-basis
$\Gamma'_W := \{\C_w : w \in W \}$ is the \emph{lower canonical basis} of $\H(W)$.

The coefficients of the lower canonical basis in terms of the standard basis are the \emph{Kazhdan-Lusztig polynomials} $P'_{x,w}$:
\be \C_w = \sum_{x \in W} P'_{x,w} T_x. \ee
(Our $P'_{x,w}$ are equal to $q^{(\ell(x)-\ell(w))/2}P_{x,w}$, where $P_{x,w}$ are the polynomials defined in \cite{KL} and $q^{1/2} = \u$.)
Now let $\mu(x,w) \in \ZZ$ be the coefficient of $\ui$ in $P'_{x,w}$ (resp. $P'_{w,x}$) if $x \leq w$ (resp. $w \leq x$).
Then the right regular representation in terms of the canonical bases of $\H_r$ takes the following simple forms:
\begin{equation}\label{e prime C on prime canbas}
\C_w \C_s =
\left\{\begin{array}{ll} [2] \C_w & \text{if}\ s \in R(w),\\
\displaystyle\sum_{\substack{\{w' \in W: s \in R(w')\}}} \mu(w',w)\C_{w'} & \text{if}\ s \notin R(w).
\end{array}\right.
\end{equation}
\begin{equation}\label{e C on canbas}
C_w C_s =
\left\{\begin{array}{ll} -[2] C_w & \text{if}\ s \in R(w),\\
\displaystyle\sum_{\substack{\{w' \in  W: s \in R(w')\}}} \mu(w',w)C_{w'} & \text{if}\ s \notin R(w).
\end{array}\right.
\end{equation}

The simplicity and sparsity of this action along with the fact that the right cells of $\Gamma_W$ and $\Gamma'_W$ often give rise to  $\CC(\u) \tsr_\mathbf{A} \H(W)$-irreducibles are among the most amazing and useful properties of canonical bases.

We will make use of the following positivity result due to Kazhdan-Lusztig and Beilinson-Bernstein-Deligne-Gabber (see, for instance, \cite{L2}).
\begin{theorem}
\label{t positive coefficients}
If $(W,S)$ is crystallographic, then the integers $\mu(x,w)$ are nonnegative.
\end{theorem}

%Let $\alpha: \mathbf{A} \to \mathbf{A}$ is the automorphism determined by $\u \mapsto -\ui$.  The coefficients of the $C$'s in terms of the $T$'s are
%\be C_w = \sum_{x \in W} \alpha(P'_{x,w}) T_x, \ee

\subsection{Cells}
\label{ss cells}
We define cells in the general setting of modules with basis, as in \cite{BProjected} (this is similar to the notion of cells of Coxeter groups from \cite{KL}).
%If $(M,\Gamma)$ is an $H$-module with basis, then we use the phrase \emph{$H$-cells of $\Gamma$} to emphasize that we are considering the action of $H$ on $M$ (and not some other algebra).

Let  $H$ be an $R$-algebra for some commutative ring $R$.
Let $M$ be a right $H$-module and $\Gamma$ an  $R$-basis of  $M$. The preorder $\klo{\Gamma}$ (also denoted $\klo{M}$) on the vertex set $\Gamma$ is generated by the relations
\be
\label{e preorder}
\delta\klocov{\Gamma}\gamma \begin{array}{c}\text{if there is an $h\in H$ such that $\delta$ appears with nonzero}\\ \text{coefficient in the expansion of $\gamma h$
in the basis $\Gamma$}. \end{array}
\ee

Equivalence classes of $\klo{\Gamma}$ are the \emph{right cells} of $(M, \Gamma)$. The preorder $\klo{M}$ induces a partial order on the right cells of $M$, which is also denoted $\klo{M}$.
We say that the right cells  $\Lambda$ and  $\Lambda'$ are isomorphic if $(R \Lambda, \Lambda)$ and $(R \Lambda', \Lambda')$ are isomorphic as modules with basis.
Sometimes we speak of the right cells of $M$ or right cells of $\Gamma$ if the pair $(M, \Gamma)$ is clear from context.
We also use the terminology \emph{right $H$-cells} when we want to make it clear that the algebra $H$ is acting.

\subsection{Cells and tableaux} \label{ss cell label conventions C_Q C'_Q}
Let  $\H_r = \H(\S_r)$ be the type $A$ Hecke algebra.  For the remainder of the paper, set $S := \{s_1,\ldots,s_{r-1}\}$ and  $J := \{s_1,\ldots,s_{r-2}\}$.

It is well known that $\field \H_r := \field \tsr_\mathbf{A} \H_r$ is semisimple and its irreducibles in bijection with partitions of $r$; let $M_\lambda$ and $M_\lambda^{\mathbf{A}}$ be the $\field\H_r$-irreducible and Specht module of $\H_r$ of shape $\lambda \vdash r$ (hence $M_\lambda \cong \field\tsr_\mathbf{A} M_\lambda^\mathbf{A}$). For any $\field \H_r$-module $N$ and partition $\lambda$ of $r$, let  $p_{M_\lambda} : N \to N$ be the $\field \H_r$-module projector with image the $M_\lambda$-isotypic component of $N$.

The work of Kazhdan and Lusztig \cite{KL} shows that the decomposition of $\Gamma_{\S_r}$ into right cells is
$\Gamma_{\S_r} = \bigsqcup_{\lambda \vdash r,\, P \in \text{SYT}(\lambda)} \Gamma_P$, where $\Gamma_P := \{C_w : P(w) = P\}$.  Moreover, the right cells $\{ \Gamma_P : \sh(P) = \lambda\}$ are all isomorphic, and, denoting any of these cells by $\Gamma_\lambda$, $\mathbf{A}\Gamma_\lambda \cong M_\lambda^\mathbf{A}$. Similarly, the decomposition of $\Gamma'_{\S_r}$ into right cells is
$\Gamma'_{\S_r} = \bigsqcup_{\lambda \vdash r,\, P \in \text{SYT}(\lambda)} \Gamma'_P$, where $\Gamma'_P := \{\C_w : \transpose{P(w)} = P\}$.  Moreover, the right cells $\{ \Gamma'_P : \sh(P) = \lambda\}$ are all isomorphic, and, denoting any of these cells by $\Gamma'_\lambda$, $\mathbf{A}\Gamma'_\lambda \cong M_\lambda^\mathbf{A}$.
A combinatorial discussion of left cells in type $A$ is given in \cite[\textsection 4]{B0}.

We refer to the basis $\Gamma_\lambda$ of $M_\lambda^\mathbf{A}$ as the \emph{upper canonical basis of $M_\lambda$} and denote it by $\{ C_Q : Q \in \text{SYT}(\lambda) \}$, where $C_Q$ corresponds to $C_w$ for any (every) $w \in \S_r$ with recording tableau $Q$. Similarly, the basis $\Gamma'_\lambda$ of $M_\lambda^\mathbf{A}$ is the \emph{lower canonical basis of $M_\lambda$}, denoted $\{ \C_Q : Q \in \text{SYT}(\lambda) \}$, where $\C_Q$ corresponds to $\C_w$ for any (every) $w \in \S_r$ with recording tableau $\transpose{Q}$. Note that with these labels the action of $C_s$ on the upper canonical basis of $M_\lambda$ is similar to \eqref{e C on canbas}, with  $\mu(Q',Q):= \mu(w',w)$ for any  $w',w$ such that  $P(w')=P(w)$, $Q' = Q(w'),Q = Q(w)$, and right descent sets
\be
R(C_Q) = \{ s_i : i + 1 \text{ is strictly to the south of $i$ in $Q$}\}.
\ee
Similarly, the action of $\C_s$ on $\{ \C_Q : Q \in \text{SYT}(\lambda) \}$ is similar to \eqref{e prime C on prime canbas}, with  $\mu(Q',Q):= \mu(w',w)$ for any  $w',w$ such that  $\transpose{P(w')}=\transpose{P(w)}$, $Q' = \transpose{Q(w')},Q = \transpose{Q(w)}$, and right descent sets
\be
R(\C_Q) = \{ s_i : i + 1 \text{ is strictly to the east of $i$ in $Q$}\}.
\ee
See Figure \ref{f Wgraph} for a picture of $\Gamma_{(3,2)}$ and  $\Gamma'_{(3,2)}$ and right descent sets.

\subsection{Dual equivalence graphs}
\label{ss dual equivalence}
To work with canonical bases in the proof of Theorem \ref{t nsH irreducibles two row case}, we make use the notion of dual equivalence graphs\footnote{We use a slightly simplified version of the dual equivalence graphs from \cite{Sami}.} from \cite{Sami}.
Given $T, T' \in \text{SYT}(\lambda)$, we say that $T$ and $T'$ are related by a \emph{dual Knuth transformation} at $i$ if
\begin{list}{(\arabic{ctr})} {\usecounter{ctr} \setlength{\itemsep}{1pt} \setlength{\topsep}{2pt}}
\item $|R(\C_T) \cap \{ s_{i - 1}, s_i \}| = |R(\C_{T'}) \cap \{ s_{i - 1}, s_i \}| = 1,$
\item $T'$ is obtained from $T$ by swapping the entries $i$ and $i+1$ in $T$ or by swapping the entries $i-1$ and $i$ in $T$.
\end{list}
If $T$ and $T'$ are related by a dual Knuth transformation at $i$, then we also say that there is a \emph{$\text{DKT}_i$-edge} between $T$ and $T'$ and write $T \dkt{i} T'$. We write $T \dkt{} T'$ if $T \dkt{i} T'$ for some $i$, $2 \leq i \leq r-1$.

Define the \emph{dual equivalence graph (DE graph)} on $\text{SYT}(\lambda)$ to be the graph with vertex set $\text{SYT}(\lambda)$ and edges given by the $\text{DKT}_i$-edges for all $i$, $2 \leq i \leq r-1$.
%We say that there is a $DKE$-edge between $T$ and $T'$ if there is a $\text{DKT}_i$-edge for some $i \in [2, r-1]$.
%See Figure \ref{f DE graph} for the example $\lambda = (3,2)$.

We will freely use the result from \cite{KL} that $T \dkt{} T'$ implies $\mu(T',T)= \mu(T,T') = 1$.
Note that this means that the $\S_r$-graph on $\Gamma'_\lambda$ (or $\Gamma_\lambda$) contains the underlying simple graph of the DE graph on $\text{SYT}(\lambda)$ (compare Figures \ref{f Wgraph} and \ref{f DE graph}).

\setlength{\cellsize}{2.2ex}
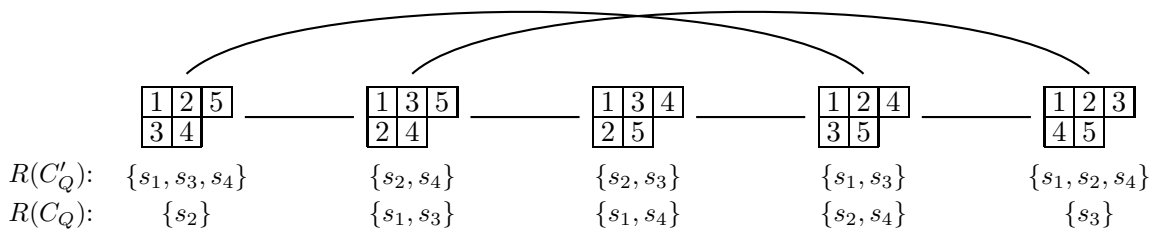
\begin{figure}
\begin{tikzpicture}[xscale = 3]
\tikzstyle{vertex}=[inner sep=0pt, outer sep=5pt, fill = white]
\tikzstyle{edge} = [draw, thick, -,black]
\tikzstyle{LabelStyleB} = [text=black, anchor=north]
\tikzstyle{LabelStyleA} = [text=black, anchor=south]

\node[vertex] (h1) at (4,0) {${\small\tableau{1&2&3 \\ 4&5}}$};
\node[vertex] (h2) at (3,0) {${\small\tableau{1&2&4 \\ 3&5}}$};
\node[vertex] (h3) at (2,0) {${\small\tableau{1&3&4 \\ 2&5}}$};
\node[vertex] (h4) at (1,0) {${\small\tableau{1&3&5 \\ 2&4}}$};
\node[vertex] (h5) at (0,0) {${\small\tableau{1&2&5 \\ 3&4}}$};

\node[vertex] (ll) at (-.6,-.8) {{\footnotesize$ R(\C_Q) $:}};
\node[vertex] (l1) at (4,-.8) {{\footnotesize$\{s_1,s_2,s_4\} $}};
\node[vertex] (l2) at (3,-.8) {{\footnotesize$ \{s_1,s_3\} $}};
\node[vertex] (l3) at (2,-.8) {{\footnotesize$ \{s_2,s_3\} $}};
\node[vertex] (l4) at (1,-.8) {{\footnotesize$ \{s_2,s_4\} $}};
\node[vertex] (l5) at (0,-.8) {{\footnotesize$ \{s_1,s_3,s_4\} $}};

\node[vertex] (ss) at (-.6, -1.3) {{\footnotesize$ R(C_Q) $:}};
\node[vertex] (s1) at (4, -1.3) {{\footnotesize$ \{s_3\} $}};
\node[vertex] (s2) at (3, -1.3) {{\footnotesize$ \{s_2, s_4\} $}};
\node[vertex] (s3) at (2, -1.3) {{\footnotesize$ \{s_1,s_4\} $}};
\node[vertex] (s4) at (1, -1.3) {{\footnotesize$ \{s_1,s_3\} $}};
\node[vertex] (s5) at (0, -1.3) {{\footnotesize$ \{s_2\} $}};

\draw[edge] (h1) to (h2);
\draw[edge] (h2) to (h3);
\draw[edge] (h3) to (h4);
\draw[edge] (h4) to (h5);
\draw[edge, bend right=70] (h1.90) to (h4.90);
\draw[edge, bend right=70] (h2.90) to (h5.90);

\end{tikzpicture}
\caption{The $\S_r$-graph on $\Gamma'_\lambda$ and $\Gamma_\lambda$. The presence (resp. absence) of an edge means that $\mu(Q',Q) = \mu(Q,Q')$ is 1 (resp. 0).}
\label{f Wgraph}
\end{figure}

\begin{figure}
\begin{tikzpicture}[xscale = 3]
\tikzstyle{vertex}=[inner sep=0pt, outer sep=5pt, fill = white]
\tikzstyle{edge} = [draw, thick, -,black]
\tikzstyle{LabelStyleB} = [text=black, anchor=north]
\tikzstyle{LabelStyleA} = [text=black, anchor=south]

\node[vertex] (h1) at (4,0) {${\small\tableau{1&2&3 \\ 4&5}}$};
\node[vertex] (h2) at (3,0) {${\small\tableau{1&2&4 \\ 3&5}}$};
\node[vertex] (h3) at (2,0) {${\small\tableau{1&3&4 \\ 2&5}}$};
\node[vertex] (h4) at (1,0) {${\small\tableau{1&3&5 \\ 2&4}}$};
\node[vertex] (h5) at (0,0) {${\small\tableau{1&2&5 \\ 3&4}}$};

\draw[edge] (h1.185) to node[LabelStyleB]{\small 3} (h2.-5);
\draw[edge] (h1.-185) to node[LabelStyleA]{\small 4} (h2.5);
\draw[edge] (h2) to node[LabelStyleB]{\small 2} (h3);
\draw[edge] (h3) to node[LabelStyleB]{\small 4} (h4);
\draw[edge] (h4.185) to node[LabelStyleB]{\small 2} (h5.-5);
\draw[edge] (h4.-185) to node[LabelStyleA]{\small 3} (h5.5);

\end{tikzpicture}
\caption{The DE graph on $\text{SYT}((3,2))$.}
\label{f DE graph}
\end{figure}
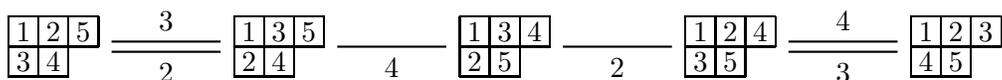

It is easy to see that (with the help of Figure \ref{f dkt example})
\be \label{e DKE complete graph}
\parbox{14cm}{for any distinct $i, j \in [k_\lambda]$, there exists at least one edge $T \dkt{r-1} T'$ in the DE graph on $\text{SYT}(\lambda)$ with $T_{a_i} = r = T'_{a_j}$ and $T_{a_j} = r-1 = T'_{a_i}$.}
\ee
Here $T_{a_i}$ denotes the entry of $T$ in the square $a_i$ (see \textsection\ref{ss type A combinatorics preliminaries}).

\begin{figure}[H]
\begin{tikzpicture}[xscale = 8.1, yscale = 2]
\tikzstyle{vertex}=[inner sep=0pt, outer sep=5pt, fill = white]
\tikzstyle{edge} = [draw, thick, -,black]
\tikzstyle{LabelStyleB} = [text=black, anchor=north]
\tikzstyle{LabelStyleA} = [text=black, anchor=south]

%\hoogte=18pt
\breedte=17pt;
\node[vertex] (d1) at (0,0) {
\tiny \Yboxdim20pt \begin{Young}\ & \ & \ & \ & \ & \ & \ & \ & \ \cr
 \ & \ & \ & \ & \ & \ & \ \cr
 \ & \ & \ & \ \cr
 \ & \ & $r \ngb 2$ & $r\ngb1$ \cr
 \ & $r$ \cr
 \end{Young}};
\node[vertex] (d2) at (1,0){
\tiny \begin{Young}\ & \ & \ & \ & \ & \ & \ & \ & \ \cr
 \ & \ & \ & \ & \ & \ & \ \cr
 \ & \ & \ & \ \cr
 \ & \ & $r\ngb2$ & $r$ \cr
 \ & $r \ngb 1$ \cr
 \end{Young}};
\node[vertex] (label1) at (0,-1) {$T$};
\node[vertex] (label2) at (1,-1) {$T'$};

\draw[edge] (d1) to node[LabelStyleB]{\tiny $r-1$} (d2);

\end{tikzpicture}
\caption{An edge of the DE graph on $\text{SYT}(\lambda)$ as in \eqref{e DKE complete graph} for $i = 1, j = 2$.}
\label{f dkt example}
\end{figure}
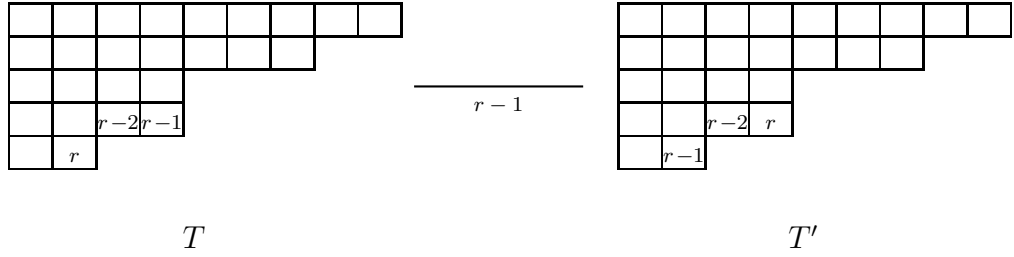

\subsection{Projected canonical bases}
\label{ss lifts}
Here we recall some results from \cite{BProjected} about projecting the upper and lower canonical bases of $M_\lambda$ onto the $\field \H_{r-1}$-irreducible isotypic components of  $M_\lambda$. These results will make it fairly easy to work with these bases in the proof of Theorem \ref{t nsH irreducibles two row case}.

For any $L \subseteq S$, define $(\liftCup_Q)^L$ to be the projection of $C_Q$ onto the irreducible $\field \H_L$-module corresponding to the right cell of $\Res_{\field \H_L} \field\Gamma_{\lambda}$ containing $C_Q$, where $\H_L$ denotes the parabolic sub-Hecke algebra of $\H_r$ with $\mathbf{A}$-basis $\{T_w:\ w\in (S_r)_L\}$.
Define $(\liftCp_Q)^L$ similarly.  If $L = J := \{s_1, \dots, s_{r-2} \}$, then by \cite[\textsection 4]{B0},  $(\liftCup_Q)^J$ (resp.  $(\liftCp_Q)^J$) is equal to $p_{M_\mu}(C_Q)$ (resp. $p_{M_\mu}(\C_Q)$), where $\mu = \sh(Q|_{[r-1]})$ and $p_{M_\mu}$ is defined in \textsection\ref{ss cell label conventions C_Q C'_Q}. Here, for a tableau  $Q$ and set $Z \subseteq \ZZ$, $Q|_Z$ denotes the subtableau of $Q$ obtained by removing the entries not in $Z$.

Maintain the notation  of \eqref{e ai definition} for the outer corners of $\lambda$.
Define a partial order $\ldneq_r$ on SYT$(\lambda)$ by declaring $Q' \ldneq_r Q$ whenever $\sh(Q'|_{[r-1]}) \gdneq \sh(Q|_{[r-1]})$.
Recall that $\field_0$ (resp. $\field_\infty$) is the subring of $\field$ consisting of rational functions with no pole at $\u = 0$ (resp. $\u = \infty$).

% Copied from canbasliftsFeb11.tex: some minor changes
\begin{lemma}[\cite{BProjected}\footnote{Lemma 7.4 of \cite{BProjected} uses a different partial order, but the proof given for this lemma also works for the partial order $\ldneq_r$ defined here.}] \label{l lift transition matrix}
The transition matrix expressing the projected basis $\{(\liftCup_Q)^J : Q \in \text{SYT}(\lambda) \}$ in terms of the upper canonical basis of $M_\lambda$ is lower-unitriangular and is the identity at  $\u = 0$ and  $\u = \infty$ (i.e. $(\liftCup_Q)^J = C_Q + \sum_{Q' \gdneq_r Q} m_{Q' Q} C_{Q'}$, $m_{Q' Q} \in \u\field_0 \cap \ui \field_\infty$). The transition matrix expressing the projected basis $\{(\liftCp_Q)^J : Q \in \text{SYT}(\lambda) \}$ in terms of the lower canonical basis of $M_\lambda$ satisfies the same properties except is upper-unitriangular instead of lower-unitriangular (i.e. $(\liftCp_Q)^J = \C_Q + \sum_{Q' \ldneq_r  Q} m'_{Q' Q} \C_{Q'}$ $m'_{Q' Q} \in \u\field_0 \cap \ui \field_\infty$).
\end{lemma}

By \cite[\textsection4]{B0}, the $\H_J$-module with basis $(\Res_{\H_J} M_\lambda, \Gamma'_\lambda)$ decomposes into right cells as
\[\Gamma'_\lambda = \bigsqcup_{i \in [k_\lambda]} \{\C_Q : \sh(Q|_{[r-1]}) = \lambda - a_i\},\]
and moreover, $\{\C_Q : \sh(Q|_{[r-1]}) = \lambda - a_i\} \xrightarrow{\cong}  \Gamma'_{\lambda-a_i}, \ \C_Q \mapsto \C_{Q|_{[r-1]}}$ is an isomorphism of right  $\H_J$-cells.
%???useful this corollary is not really needed-- the lemma above could be applied instead, however I like this statement better
\begin{corollary}\label{c geck relative a invariant}
Let $\klo{\Res_{J} \Gamma'_\lambda}$ be the partial order on the right cells of the $\H_J$-module with basis $(\Res_{\H_J} M_\lambda, \Gamma'_\lambda)$. This partial order is a total order with $\Gamma'_{\lambda - a_i} \klo{\Res_{J}\Gamma'_\lambda} \Gamma'_{\lambda - a_j}$ exactly when $i \leq j$. Similarly, $(\Res_{\H_J} M_\lambda, \Gamma_\lambda)$ has a right cell isomorphic to  $\Gamma_{\lambda-a_i}$ for each  $i \in [k_\lambda]$ and the partial order $\klo{\Res_{J}\Gamma_\lambda}$ on right cells is a total order with $\Gamma_{\lambda - a_i} \klo{\Res_{J}\Gamma_\lambda} \Gamma_{\lambda - a_j}$ exactly when $i \geq j$.
\end{corollary}
\begin{proof}
Lemma \ref{l lift transition matrix} shows that $\Gamma'_{\lambda - a_i} \klo{\Res_{J}\Gamma'_\lambda} \Gamma'_{\lambda - a_j}$ implies $i \leq j$.  To prove the converse, it suffices to show the existence of certain nonzero $\mu(Q',Q)$.  The $\text{DKT}_i$-edges from \eqref{e DKE complete graph} suffice.
\end{proof}

%Copied from canbasliftsFeb11.tex
%\begin{theorem}
%\label{t transition C' to C}
%The transition matrix $S(\lambda) = (S_{Q' Q})_{Q', Q \in \text{SYT}(\lambda)}$ expressing the lower canonical basis of $ M_\lambda$ in terms of the upper canonical basis of $M_\lambda$ (i.e. $\C_Q = \sum_{Q' \in \text{SYT}(\lambda)} S_{Q' Q} C_{Q'})$ has $\br{\cdot}$-invariant entries that belong to $\field_0 \cap \field_\infty$ and is the identity matrix at $\u = 0$ and $\u = \infty$.
%\end{theorem}
We will also need the following theorem, one of the main results of \cite{BProjected}.
\begin{theorem}[\cite{BProjected}]
\label{t transition C' to C}
The transition matrix expressing the lower canonical basis $\{ \C_Q : Q \in \text{SYT}(\lambda) \}$ of $M_\lambda$ in terms of the upper canonical basis $\{ C_Q : Q \in \text{SYT}(\lambda) \}$ of $M_\lambda$ has entries belonging to $\field_0 \cap \field_\infty$ and is the identity matrix at $\u = 0$ and $\u = \infty$.
\end{theorem}
See \cite[Example 7.5]{BProjected} for an example of this transition matrix.  One consequence of this theorem is that $\field_0 \Gamma'_\lambda = \field_0 \Gamma_\lambda$.  Let  $\L_\lambda$ denote this  $\field_0$-lattice.

%??needed? I don't think I need this anywhere.
%The localization $\field_0 $ is a discrete valuation ring with valuation  $v: \field_0 \to \ZZ_{\geq 0} \cup \{\infty\}$, with $v(x)$ equal to the largest $k$ such that  $x \in (\u^k)$ and equal to  $\infty$ if no such $k$ exists (equivalently, if $x = 0$).
%%Given  $a \in \field_0$, let  $p(a)$ be the image of  $a$ in  $\field_0/ \u \field_0 \cong \QQ$.

\begin{lemma}[The projection lemma]\label{l projections are not too tricky}
Fix  $i \in [k_\lambda]$.  Let $x = \sum_{Q \in \text{SYT}(\lambda)} a_Q \C_Q$ be an element of $M_\lambda$ such that for each $Q$ with $\sh(Q|_{[r-1]}) = \lambda-a_j$ and $j \geq i$, there holds $a_Q \in \field_0$.
Then
\[ p_{M_{\lambda-a_i}}(x) \equiv \sum_{\substack{ Q \in \text{SYT}(\lambda), \\ Q_{a_i} = r}} a_Q (\liftCp_Q)^J \mod \u \L_{\lambda-a_i}. \]

Similarly, if $x = \sum_{Q \in \text{SYT}(\lambda)} a_Q C_Q$ is an element of $M_\lambda$ such that for each $Q$  with $\sh(Q|_{[r-1]}) = \lambda-a_j$ and $j \leq i$, there holds $a_Q \in \field_0$, then
\[ p_{M_{\lambda-a_i}}(x) \equiv \sum_{\substack{ Q \in \text{SYT}(\lambda), \\ Q_{a_i} = r}} a_Q (\liftCup_Q)^J \mod \u \L_{\lambda-a_i}. \]
\end{lemma}
\begin{proof}
This follows easily from Lemma \ref{l lift transition matrix}.
\end{proof}

\section{The nonstandard Hecke algebra $\nsH_r$}
\label{s the nonstandard Hecke algebra}
The nonstandard Hecke algebra  was introduced in \cite{GCT4} to study the Kronecker problem. Its role in the Kronecker problem is discussed in  \cite[\textsection1]{BMSGCT4} and \cite{canonical};  some of its representation theory is discussed in \cite[\textsection11]{BMSGCT4} and \cite{GCT7}, including a complete description  $\field \nsH_{3}$ and  $\field \nsH_{4}$-irreducibles; the problem of constructing a canonical basis for  $\nsH_r$ is discussed and \cite[\textsection19]{BMSGCT4} and \cite{canonical}.  The main purpose of this paper is to determine the irreducibles of the nonstandard Temperley-Lieb algebra $ \field \nsH_{r,2}$, which is a quotient algebra of  $\field \nsH_{r}$.  Here we assemble some basic facts about $\nsH_r$ from \cite{Bnsbraid, GCT4, BMSGCT4} and prove a few new ones.

%The simplest lower and upper Kazhdan-Lusztig basis elements $\C_s = T_s + \u$ and  $C_s = T_s - \ui$,  $s\in S$ are proportional to the primitive central idempotents of  $\Ap (\H_r)_{\{s\}} \cong \Ap \H_2$, where  $\Ap := \mathbf{A}[\frac{1}{[2]}]$.  Specifically,
%$\frac{1}{[2]} \C_{s_1}$ (resp. $- \frac{1}{[2]} C_{s_1}$) is the  idempotent corresponding to the trivial (resp. sign) representation of $\Ap \H_2$.

\subsection{Definition of $\nsH_r$}
Recall that $S$ is now defined to be $\{s_1,\ldots,s_{r-1}\}$.  We repeat the definition of $\nsH_r$ from the introduction:
\begin{definition}
\label{d nonstandard Hecke algebra}
The \emph{type $A$ nonstandard Hecke algebra}  $\nsH_r$ is the subalgebra of $\H_r \tsrvw \H_r$ generated by the elements
\be \label{e sP definition}
\sP_s := \C_s \tsrvw \C_s + C_s \tsrvw C_s, \ s \in S.
\ee
We let $\nsbr{\Delta}:\nsH_r \hookrightarrow   \H_r \tsrvw \H_r$ denote the canonical inclusion, which we think of as a deformation of the coproduct
$\Delta_{\ZZ \S_r} :\ZZ \S_r \to \ZZ \S_r \tsrvw \ZZ \S_r$,  $w \mapsto w \tsrvw w$.
\end{definition}
The nonstandard Hecke algebra is also the subalgebra of $\H_r \tsrvw \H_r$ generated by
\[ \sQ_s := [2]^2 - \sP_s = - \C_s \tsrvw C_s - C_s \tsrvw \C_s, \ s \in S. \]
We will write $\sP_{i}$ (resp.  $\sQ_i$) as shorthand for  $\sP_{s_{i}}$ (resp.  $\sQ_{s_i}$), $i \in [r-1]$.
For a ring homomorphism $\field \to \mathbf{A}$, we have the specialization $\field \nsH_r := \field \tsr_{\mathbf{A}} \nsH_r$ of the nonstandard Hecke algebra.

%Let $\Delta : \nsH_r \hookrightarrow \H_r \tsr \H_r$ be the inclusion. Define $p_+$ (resp. $p_-$) to be the idempotent corresponding to the trivial representation of $\H_2$. Note that $p_+ = \frac{1}{[2]} \C_1$ and $p_- = - \frac{1}{[2]} C_1$.
%?? Defined this earlier perhaps?

The elements $\sP_i$ and  $\sQ_i$ satisfy the quadratic relations $\sP_i^2 = [2]^2 \sP_i$ and  $\sQ_i^2 = \f \sQ_i$, and  $\sP_i$ and  $\sP_{i+1}$ satisfy a nonstandard version of the braid relation  (see \cite{Bnsbraid}).  For  $r \geq 4$, the $\sP_i$ satisfy additional relations which seem to be extremely difficult to describe (see \cite{GCT4}).
\subsection{Representation theory of $S^2\H_r$}
\label{ss representation theory of S2Hr}
The representations of  $\nsH_r$ are related to those of $S^2 \H_r$ by the fact that $\nsH_r \subseteq S^2 \H_r$ (see, e.g., \cite[Proposition 11.6]{BMSGCT4}), so any $S^2 \H_r$-module is an $\nsH_r$-module by restriction.
We recall the description of the $\field S^2\H_r$-irreducibles from \cite{BMSGCT4}.  These irreducibles are close to those of $\field \nsH_r$, and even closer to
those of $\field \nsH_{r,2}$, which will be described in  \textsection\ref{s irreducibles of nshr2}.

First note that we have the following commutativity property for any $\H_r$-modules  $M$ and  $M'$:
\be
\label{e commutativity S2H}
\Res_{S^2\H_r} M \tsrvw M' \cong \Res_{S^2\H_r} M' \tsrvw M,
\ee
where the isomorphism is given by the flip $\tau$, $\tau(a \tsr b) = b \tsr a$.

Recall from  \textsection\ref{ss type A combinatorics preliminaries} that $\mathscr{P}_{r}$ denotes the set of partitions of size $r$ and $\mathscr{P}'_{r}$ the set of  partitions of  $r$ that are not a single row or column shape.
%Recall from \textsection\ref{ss cell label conventions C_Q C'_Q} that $M^\mathbf{A}_\lambda$ denotes the Specht module of $\H_r$ so that $M_\lambda \cong \field\tsr_\mathbf{A} M_\lambda^\mathbf{A}$.
\begin{propdef}[\cite{BMSGCT4}]
\label{p S2Hr representations}
Define the following  $S^2 \H_r$-modules.  After tensoring these with $\field$, this is the list of distinct $\field S^2\H_r$-irreducibles
\begin{list}{\emph{(\arabic{ctr})}} {\usecounter{ctr} \setlength{\itemsep}{1pt} \setlength{\topsep}{2pt}}
\item $M^\mathbf{A}_{\{\lambda,\mu\}} := \Res_{S^2\H_r} M^\mathbf{A}_\lambda \tsrvw M^\mathbf{A}_\mu$, $\{\lambda,\mu\}\subseteq \mathscr{P}_r$, $\lambda\neq\mu$,
\item $S^2 M^\mathbf{A}_\lambda := \Res_{ S^2\H_r} S^2 M^\mathbf{A}_\lambda$, $\lambda\in\mathscr{P}_r$,
\item $\Wedge^2 M^\mathbf{A}_\lambda := \Res_{ S^2\H_r} \Wedge^2 M^\mathbf{A}_\lambda$, $\lambda\in\mathscr{P}'_r$.
\end{list}
Let $ M_{\{\lambda,\mu\}}$, $ S^2 M_\lambda$,
 $ \Wedge^2 M_\lambda$ denote the corresponding $\field S^2 \H_r$-modules.
\end{propdef}

\subsection{Contragradients of $\H_r$-modules}
Any anti-automorphism $S$ of an $\mathbf{A}$-algebra $H$ allows us to define contragradients of $H$-modules: let $\langle \cdot, \cdot \rangle: M \tsr M^* \to \mathbf{A}$ be the canonical pairing, where $M^*$ is the $\mathbf{A}$-module  $\hom_{\mathbf{A}}(M,\mathbf{A})$. Then the $H$-module structure on  $M^*$ is defined by
\[ \langle m, m'h\rangle = \langle m S(h),m' \rangle \text{ for any } h \in H,\ m \in M, m' \in M^*.\]

There is an $\mathbf{A}$-algebra automorphism $\theta : \H_r \to \H_r$ defined by $\theta(T_s) = - T_s^{-1},\ s \in S$.
It is not hard to show that $\theta$ is an involution and satisfies
$\theta(\C_w) = (-1)^{\ell(w)} C_w.$
Let $1^{\text{op}}$ be the $\mathbf{A}$-anti-automorphism of $\H_r$ given by $1^{\text{op}}(T_w) = T_{w^{-1}}$.
Let $\theta^{\text{op}}$ be the $\mathbf{A}$-anti-automorphism of $\H_r$ given by $\theta^{\text{op}} = \theta \circ 1^{\text{op}} = 1^{\text{op}} \circ \theta$.

%???needed?
%There are anti-automorphisms $1^{\text{op}} := 1^{\text{op}} \tsrvw 1^{\text{op}}$ and $(\Theta)^\text{op} := 1^{\text{op}} \circ \Theta$ of $\nsH_r$, where $\Theta$ is defined in Proposition \ref{p theta on nsH} (iii).
%For an $\H_r$-module $M$, we write $M^\dualone$ (resp.  $M^\dualtheta$) for the contragradient of $M$ corresponding to the anti-automorphism  $1^\text{op}$ (resp.  $\theta^\text{op}$).
%For an $\nsH_r$-module $\nsbr{M}$, we also write $\nsbr{M}^\dualone$ (resp.  $\nsbr{M}^\dualtheta$) for the contragradient of $\nsbr{M}$ corresponding to the anti-automorphism  $1^\text{op}$ (resp.  $\Theta^\text{op}$).

Let  $\{ \dual{\C_w} : w \in \S_r\} \subseteq \hom_{\mathbf{A}}(\H_r,\mathbf{A})$ be the basis dual to  $\{\C_w: w \in \S_r\}$. Let $w_0$ be the longest element of $\S_r$.

Let  $Z_\lambda^*$ be the SYT of shape $\lambda$ with $1, \ldots, \lambda_1$ in the first row,  $\lambda_1+1,\ldots,\lambda_1+\lambda_2$ in the second row, etc.
For an SYT $Q$, let  $\ell(Q)$ denote the distance between $Q$ and $Z_\lambda^*$ in the DE graph on SYT$(\lambda)$.
It is not hard to show that for any $P \in \text{SYT}(\lambda)$, $\ell(Q) \equiv  \ell(w)-\ell(z)\mod 2$, where $w = \text{RSK}^{-1}(P,Q),$ $z = \text{RSK}^{-1}(P,Z_\lambda^*)$.

%??check make sure this works, because Mathas claims this only for  Hecke algebra over a field. Maybe hthe catch is for infinite-dimensional representations
%note that this proposition is not used in a substantial way in the paper, so the fact that this paper {B4} is not ready yet might be okay

%This appears in BMSGCT4, so the proposition below should contain this--it does
%\begin{proposition}[\cite{B4} (see also {\cite[Exercises 2.7, 3.14]{Mathas}})]
%\label{p dual basis to C' S = 1 theta op}
%The contragradients of the Specht module  $M^\mathbf{A}_\lambda$ of  $\H_r$ are given by
%\[
%\begin{array}{lcr}
% (M^\mathbf{A}_\lambda)^\dualone \cong M^\mathbf{A}_{\lambda} &  \text{and} & (M^\mathbf{A}_\lambda)^\dualtheta \cong M^\mathbf{A}_{\lambda'}.
%\end{array}
%\]
%\end{proposition}
\begin{proposition}\label{p dual basis to C'}
\begin{list}{\emph{(\roman{ctr})}} {\usecounter{ctr} \setlength{\itemsep}{1pt} \setlength{\topsep}{2pt}}
\item The right $\H_r$-modules $\H_r^\dualone$ and $\H_r$ are isomorphic via
\[\alpha_\dualone: \H_r^\dualone \xrightarrow{\cong} \H_r, \ \dual{\C_w} \mapsto C_{w_0 w}, \ w \in \S_r. \]
\item The right $\H_r$-modules $\H_r^\dualtheta$ and $\H_r$ are isomorphic via
%??useful it doesn't matter if \ell(w_0 w) is used here instead because this just changes the isomorphism by -1
\[\alpha_\dualtheta: \H_r^\dualtheta \xrightarrow{\cong} \H_r, \ \dual{\C_w} \mapsto (-1)^{\ell(w)} \C_{w_0 w},\ w \in \S_r. \]
\item The restriction of  $\alpha_\dualone^{-1}$ to any right cell $\Gamma_\lambda$ of  $\Gamma_{\S_r}$ yields the isomorphism
\[ M^\mathbf{A}_\lambda \xrightarrow{\cong} (M^\mathbf{A}_{\lambda})^\dualone, \ C_{Q} \mapsto \dual{\C_{Q}}, \ Q \in \text{SYT}(\lambda). \]
\item The restriction of $\alpha_\dualtheta^{-1}$ to any right cell $\Gamma'_\lambda$ of  $\Gamma'_{\S_r}$ yields, up to a sign, the isomorphism
\[ M^\mathbf{A}_{\lambda'} \xrightarrow{\cong} (M^\mathbf{A}_{\lambda})^\dualtheta, \ (-1)^{\ell(\transpose{Q})}\C_{Q} \mapsto \dual{\C_{\transpose{Q}}}, \ Q \in \text{SYT}(\lambda'). \]
\end{list}
\end{proposition}
\begin{proof}
We first record the following formulae which are immediate from (\ref{e prime C on prime canbas}), (\ref{e C on canbas}), and $R(w_0 w) = S \backslash R(w)$.
\begin{align}\label{e prime C on unprime canbas}
C_{w_0 w} \C_{s} =&
\left\{ \begin{array}{ll}
\two C_{w_0 w} + \sum_{\substack{\{w_0 w' \in \S_r : s \notin R(w')\}}} \mu(w_0 w',w_0 w) C_{w_0 w'} & \text{if}\ s \in R(w),\\
0 & \text{if}\ s \notin R(w).
\end{array} \right.\\
\label{e theta prime C on prime canbas}
\C_w \theta(\C_{s})=&
\left\{ \begin{array}{ll} 0 & \text{if}\ s \in R(w),\\
\two \C_w - \sum_{\substack{\{w' \in \S_r : s \in R(w')\}}} \mu(w',w)\C_{w'} & \text{if}\ s \notin R(w).
\end{array} \right.
\end{align}

By the definition of $\H_r^\dualone$,
\begin{equation}
\dual{\C_w} \C_{s}=
\left\{\begin{array}{ll}
[2] \dual{\C_w} + \sum_{\substack{\{w' \in \S_r : s \notin R(w')\}}} \mu(w,w') \dual{\C_{w'}} & \text{if}\ s \in R(w), \\
0 & \text{if}\ s \notin R(w).
\end{array}\right.
\end{equation}
Statement (i) then follows from (\ref{e prime C on unprime canbas}) as $\mu(w,w') = \mu(w_0 w',w_0 w)$ \cite[Corollary 3.2]{KL}.

By the definition of $\H_r^\dualtheta$ and from (\ref{e theta prime C on prime canbas}), we obtain
\begin{equation}
\dual{\C_w} \C_{s}=
\left\{\begin{array}{ll}
[2] \dual{\C_w} & \text{if}\ s \notin R(w),\\
-\sum_{\substack{\{w' \in \S_r : s \notin R(w')\}}} \mu(w,w') \dual{\C_{w'}} & \text{if}\ s \in R(w).
\end{array}\right.
\end{equation}
Statement (ii) then follows from \eqref{e prime C on prime canbas} using $R(w_0 w) = S \backslash R(w)$, $\mu(w,w') = \mu(w_0 w',w_0 w)$, and the fact that $\mu(w',w) \neq 0$ implies  $(-1)^{\ell(w')} = - (-1)^{\ell(w)}$.

Statements (iii) and (iv) then follow from (i) and (ii), respectively, the fact that $Q(w_0 w) = \transpose{Q(w)}$ (see, e.g., \cite[A1.2]{F}), and the definitions in \textsection\ref{ss cell label conventions C_Q C'_Q}.
\end{proof}

As discussed in \cite{Bnsbraid, BMSGCT4}, the inclusion $\nsbr{\Delta}:\nsH_r \hookrightarrow   \H_r \tsrvw \H_r$ is not a good approximation of the coproduct $\Delta_{\ZZ \S_r}$, though it is in a certain sense the closest approximation possible.  There are a couple ways that $\nsH_r$ behaves like a Hopf algebra, one of which is the following.
\begin{proposition}[\cite{Bnsbraid}] \label{p hecke algebra antipode}
The involutions $1^{\text{op}}$ and $\theta^{\text{op}}$ are antipodes in the following sense:
\begin{flalign*}
\mu \circ (1^{\text{op}} \tsr 1) \circ \nsbr{\Delta} &= \eta \circ \nsbr{\epsilon}_+,  \\
\mu \circ (\theta^{\text{op}} \tsr 1) \circ \nsbr{\Delta} &= \eta \circ \nsbr{\epsilon}_-,
\end{flalign*}
where these are equalities of maps from $\nsH_r$ to $\H_r$. Here $\mu$ is the multiplication map for $\H_r$ and $\eta : \field \to \H_r$ is the unit of $\H_r$.
\end{proposition}

\subsection{Some representation theory of  $\nsH_r$}
\label{ss some representation theory of nsH}
It is shown in \cite{BMSGCT4} (Proposition 11.8) that $\field \nsH_r$ is semisimple.

\begin{remark}
It is reasonable to suspect that  $\field\nsH_r$ is split semisimple, and indeed, our computations are consistent with this being true.  In this paper we show that the nonstandard Temperley-Lieb algebra $\field\nsH_{r,2}$ is split semisimple by explicitly determining its irreducibles.  We are curious if there is a way to show that  $ \field \nsH_{r}$ is split semisimple without explicitly determining its irreducibles.
%(even though this may be doable)
\end{remark}

There are one-dimensional trivial and sign representations of $\nsH_r$, which we denote by $\nsbr{\epsilon}_{+}$ and $\nsbr{\epsilon}_{-}$:
\[ \begin{array}{cccc}
  \nsbr{\epsilon}_{+} : \sP_{s} \mapsto \two^{2}, && \nsbr{\epsilon}_{-} : \sP_s \mapsto 0, & s \in S.
\end{array} \]

For  $\lambda, \mu \vdash r$, the $\nsH_r$-module $\Res_{\nsH_r} M^{\mathbf{A}}_\lambda \tsr M^{\mathbf{A}}_\mu \cong \Res_{\nsH_r} M^{\mathbf{A}}_\mu \tsr M^{\mathbf{A}}_\lambda$ is denoted\footnote{The more correct notation $\nsbr{M}^{\mathbf{A}}_{\{\lambda, \mu\}}$ is used in \cite{BMSGCT4}, but in this paper the shorter $\nsbr{M}^{\mathbf{A}}_{\lambda, \mu}$ is preferable for being less cumbersome.} $\nsbr{M}^{\mathbf{A}}_{\lambda, \mu}$. Let $\nssym{2}{M}^{\mathbf{A}}_\lambda$ (resp. $\nswedge{2}{M}^{\mathbf{A}}_\lambda$) denote the $\nsH_r$-module $\Res_{\nsH_r} S^2 M^{\mathbf{A}}_\lambda$ (resp. $\Res_{\nsH_r} \Wedge^2 M^{\mathbf{A}}_\lambda$), where $S^2 M^{\mathbf{A}}_\lambda$ and $\Wedge^2 M^{\mathbf{A}}_\lambda$ are as in Proposition-Definition \ref{p S2Hr representations}.
Let $\nsbr{M}_{\lambda,\mu}$, $\nssym{2}{M}_\lambda, \nswedge{2}{M}_\lambda$ denote the corresponding $\field \nsH_r$-modules.

Let
\be \label{e triv inclusion map}
\mathbf{A} \xrightarrow{I} (M^\mathbf{A}_{\lambda})^\dualone \tsrvw M^\mathbf{A}_{\lambda}
\ee
be the canonical inclusion given by sending  $1 \in \mathbf{A}$ to  $I \in \End(M^\mathbf{A}_\lambda) \cong (M^\mathbf{A}_{\lambda})^\dualone \tsrvw M^\mathbf{A}_{\lambda}$.
Let
\be \label{e triv surjection map}
M^\mathbf{A}_{\lambda} \tsrvw (M^\mathbf{A}_{\lambda})^\dualone \xrightarrow{\trace} \mathbf{A}
\ee
be the canonical surjection.

We then have the following $\nsH_r$-module homomorphisms
\begin{align*}
\nsbr{\epsilon}_+ &\xrightarrow{I} (M^\mathbf{A}_{\lambda})^\dualone \tsrvw M^\mathbf{A}_{\lambda}, \\
\ker(\trace) &\hookrightarrow  M^\mathbf{A}_{\lambda} \tsrvw (M^\mathbf{A}_{\lambda})^\dualone \xrightarrow{\trace} \nsbr{\epsilon}_+,
\end{align*}
To see this, note that in general, if  $M$ is an  $H$-module and  $H$ is a Hopf algebra with counit  $\epsilon$, then it follows from the axiom for the antipode that $\epsilon \xrightarrow{I} M^* \tsr M$ and  $M \tsr M^*  \xrightarrow{\trace} \epsilon$ are  $H$-module homomorphisms.
%???check
The same proof works in the present setting using Proposition \ref{p hecke algebra antipode} in place of the antipode axiom.
%Thus \[N \to (M^\mathbf{A}_{\lambda})^\dualone \tsrvw M^\mathbf{A}_{\lambda} \cong M^\mathbf{A}_{\lambda} \tsrvw (M^\mathbf{A}_{\lambda})^\dualone \xrightarrow{\trace} \nsbr{\epsilon}_+.\] is a split exact sequence, hence

Since $\frac{1}{|\text{SYT}(\lambda)|} \tau \circ I$ is a splitting of $\trace$ and  $(M^\mathbf{A}_{\lambda})^\dualone \cong M^\mathbf{A}_{\lambda}$ (Proposition \ref{p dual basis to C'} (iii)), we obtain the decomposition of $\nsH_r$-modules
\be  \label{e nsH trace}
\ker(\trace) \oplus \nsbr{\epsilon}_+ \cong \nsbr{M}^\mathbf{A}_{\lambda,\lambda}.
\ee
%Since $(M_\lambda|_{\u=1})^\dualone \cong M_\lambda|_{\u=1}$, it follows that  $S^2 M_\lambda|_{\u=1}$ contains a copy of the trivial $\QQ \S_r$-module.
Moreover, as a consequence of Proposition \ref{p triv in Mdual M} (i) below, $\nsbr{\epsilon}_+ \subseteq \nsbr{M}^\mathbf{A}_{\lambda, \lambda}$ lies in
$\nssym{2}{M}_\lambda^\mathbf{A}$.
Then define  $S' \nsbr{M}^\mathbf{A}_\lambda := \ker(\trace) \cap \nssym{2}{M}^\mathbf{A}_\lambda$.  The decomposition \eqref{e nsH trace} yields the decomposition
%This is not totally obvious but not hard to show either
\be  \label{e nsH S' iso}
\nssym{2}{M}^\mathbf{A}_\lambda \cong S' \nsbr{M}^\mathbf{A}_\lambda \oplus \nsbr{\epsilon}_+.
\ee

%this is from GCT four citing this paper, so make sure this follows easily from the next corollary
%\begin{proposition}[\cite{B4}]
%\label{p sign in lambda tsr lambda'}
%The sign representation $\field\nsbr{\epsilon}_-$ of $\field\nsH_r$ occurs with multiplicity one in $M_\lambda \tsrvw M_{\lambda'}$ and not at all in $M_\lambda \tsrvw M_\mu$, $\mu \neq \lambda'$. Moreover, the inclusion $\nsbr{i}_-: \field\nsbr{\epsilon}_- \hookrightarrow M_\lambda \tsrvw M_{\lambda'}$ can be expressed in terms of the basis $\Gamma_\lambda \tsrvw \Gamma_{\lambda'}$ as
%\[
%1 \mapsto \sum_{Q \in \text{SYT}(\lambda)} (-1)^{\ell(\transpose{Q})} C_Q \tsrvw C_{\transpose{Q}},
%\]
%and the surjection $\nsbr{s}_- : M_\lambda \tsrvw M_{\lambda'} \twoheadrightarrow \field\nsbr{\epsilon}_-$ by
%\[
%\sum_{Q', Q \in \text{SYT}(\lambda)} a^{Q'\transpose{Q}}C_{Q'} \tsrvw C_{\transpose{Q}} \mapsto \frac{1}{|\text{SYT}(\lambda)|} \sum_{Q \in \text{SYT}(\lambda)} (-1)^{\ell(\transpose {Q})}a^{Q\transpose{Q}} \ \ (a^{Q'\transpose{Q}} \in \field),
%\]
%where $\ell(\transpose{Q})$ is as in \textsection\ref{ss type A combinatorics preliminaries}.
%\end{proposition}
\begin{proposition}
\label{p triv in Mdual M}
The maps \eqref{e triv inclusion map} and \eqref{e triv surjection map} as well as the analogous maps for $ \nsbr{\epsilon}_-$ can be made explicit using upper and lower canonical bases:

\begin{list}{\emph{(\roman{ctr})}} {\usecounter{ctr} \setlength{\itemsep}{1pt} \setlength{\topsep}{2pt}}
\item The inclusion $\nsbr{\epsilon}_+ \hookrightarrow M_\lambda^\mathbf{A} \tsr M_\lambda^\mathbf{A}$ is given by
\[ 1 \mapsto \sum_{Q \in \text{SYT}(\lambda)} C_Q \tsr \C_Q = \sum_{Q \in \text{SYT}(\lambda)} \C_Q \tsr C_Q.\]
\item The surjection $M_\lambda^\mathbf{A} \tsr M_\lambda^\mathbf{A} \to \nsbr{\epsilon}_+$ is given by
\[ \sum_{T, U \in \text{SYT}(\lambda)} a^{T U} \C_T \tsr C_U \mapsto \textstyle \frac{1}{|\text{SYT}(\lambda)|} \displaystyle \sum_{U \in \text{SYT}(\lambda)} a^{U U}, \text{ for any $a^{T U} \in \mathbf{A}$.}\]
\item The inclusion $\nsbr{\epsilon}_- \hookrightarrow M_{\lambda'}^\mathbf{A} \tsr M_\lambda^\mathbf{A}$ is given by
\[ 1 \mapsto \sum_{Q \in \text{SYT}(\lambda)} (-1)^{\ell(Q)}\C_{\transpose{Q}} \tsr \C_{Q}.\]
\item The surjection $M_{\lambda}^\mathbf{A} \tsr M_{\lambda'}^\mathbf{A} \to \nsbr{\epsilon}_-$ is given by
\[ \sum_{T, U \in \text{SYT}(\lambda)} a^{T \transpose{U}} \C_T \tsr \C_{\transpose{U}} \mapsto \textstyle \frac{1}{|\text{SYT}(\lambda)|} \displaystyle \sum_{U \in \text{SYT}(\lambda)} (-1)^{\ell(U)} a^{U \transpose{U}}, \text{ for any $a^{T \transpose{U}} \in \mathbf{A}$.}\]
\end{list}
\end{proposition}
Note that since  $\nsbr{\epsilon}_+ \subseteq \nssym{2}{M}^\mathbf{A}_\lambda$ by (i), (ii) remains valid with  $C_U \tsr \C_T$ in place of $\C_T \tsr C_U$.
\begin{proof}
The map  $I$ of \eqref{e triv inclusion map} is given by
\[\textstyle  1 \mapsto \sum_{Q \in \text{SYT}(\lambda)} \dual{\C_Q} \tsr \C_Q.\]
Applying the isomorphism of Proposition \ref{p dual basis to C'} (iii) then yields (i), except the equality.  The equality in (i) follows from the fact that
$\tau \circ I: \nsbr{\epsilon}_+ \hookrightarrow M_\lambda^\mathbf{A} \tsr (M_\lambda^\mathbf{A})^\dualone$ is an $\nsH_r$-module homomorphism (since $\nsH_r \subseteq S^2 \H_r$), the multiplicity of  $\field \nsbr{\epsilon}_+$ in  $\nsbr{M}_{\lambda,\lambda}$ is 1, and Theorem \ref{t transition C' to C}.

The map $\trace$ of \eqref{e triv surjection map} is given by
\[\textstyle  \sum_{T, U \in \text{SYT}(\lambda)} a^{T U} \C_T \tsr \dual{\C_U} \mapsto \frac{1}{|\text{SYT}(\lambda)|} \sum_{U} a^{U U}, \text{ for any $a^{T U} \in \mathbf{A}$}, \]
so (ii) also follows from Proposition \ref{p dual basis to C'} (iii).
Statements (iii) and (iv) are proved in a similar way using Proposition \ref{p hecke algebra antipode} and Proposition \ref{p dual basis to C'} (iv).
\end{proof}

\subsection{The action of $\sP_s$ on  $M_\lambda \tsr M_\mu$}
\label{ss sP action on bases}
For the proof of the main theorem, it is convenient to record the action of $\sP_s$ on  $M_\lambda \tsr M_\mu$ in
the bases
\[ \Gamma'_\lambda \tsr \Gamma'_\mu = \{ \C_T \tsr \C_U : T \in \text{SYT}(\lambda), \ U \in \text{SYT}(\mu) \},\]
$\Gamma_\lambda \tsr \Gamma'_\mu$, and $\Gamma_\lambda \tsr \Gamma_\mu$.  These calculations are easily made using (\ref{e prime C on prime canbas}) and (\ref{e C on canbas}).
\refstepcounter{equation}
\begin{align*}
&(\C_T \tsr \C_U) \sP_s = \tag{\theequation}\label{e sP on C' C'} \\
& \ \begin{cases}
\f \C_T \tsr \C_U & \text{if } s \in R(\C_T) \text{ and } s \in R(\C_U) \\
[2] \displaystyle \sum_{s \in R(\C_{U'})} \mu(U', U) \C_{T} \tsr \C_{U'} & \text {if } s \in R(\C_T) \text{ and } s \notin R(\C_U) \\
[2] \displaystyle \sum_{s \in R(\C_{T'})} \mu(T', T) \C_{T'} \tsr \C_U & \text {if } s \notin R(\C_T) \text{ and } s \in R(\C_U) \\
\f \C_T \tsr \C_U \\
-[2] \left( \displaystyle \sum_{s \in R(\C_{T'})} \mu(T', T) \C_{T'} \tsr \C_U + \sum_{s \in R(\C_{U'})} \mu(U', U) \C_{T} \tsr \C_{U'} \right) \\
+ 2 \displaystyle \sum_{s \in R(\C_{T'}), s \in R(\C_{U'})} \mu(T', T) \mu(U', U) \C_{T'} \tsr \C_{U'} & \text{if } s \notin R(\C_U) \text{ and } s \notin R(\C_T) \\
\end{cases}
\end{align*}
\be \label{e sP on C C'}
(C_T \tsr \C_U) \sP_s =
\begin{cases}
0 & \text{if } s \in R(C_T) \text{ and } s \in R(\C_U) \\
\f C_T \tsr \C_U - [2] \sum_{s \in R(\C_{U'})} \mu(U', U) C_T \tsr \C_{U'} & \text {if } s \in R(C_T) \text{ and } s \notin R(\C_U) \\
\f C_T \tsr \C_U + [2] \sum_{s \in R(C_{T'})} \mu(T', T) C_{T'} \tsr \C_U & \text {if } s \notin R(C_T) \text{ and } s \in R(\C_U) \\
-[2] \sum_{s \in R(C_{T'})} \mu(T', T) C_{T'} \tsr \C_U \\
+[2] \sum_{s \in R(\C_{U'})} \mu(U', U) C_{T} \tsr \C_{U'} \\
+ 2 \sum_{s \in R(C_{T'}), s \in R(\C_{U'})} \mu(T', T) \mu(U', U) C_{T'} \tsr \C_{U'} & \text{if } s \notin R(C_T) \text{ and } s \notin R(\C_U) \\
\end{cases}
\ee
\refstepcounter{equation}
\begin{align*}
&(C_T \tsr C_U) \sP_s = \tag{\theequation}\label{e sP on C C} \\
& \  \begin{cases}
\f C_T \tsr C_U & \text{if } s \in R(C_T) \text{ and } s \in R(C_U) \\
-[2] \displaystyle \sum_{s \in R(C_{U'})} \mu(U', U) C_{T} \tsr C_{U'} & \text {if } s \in R(C_T) \text{ and } s \notin R(C_U) \\
-[2] \displaystyle \sum_{s \in R(C_{T'})} \mu(T', T) C_{T'} \tsr C_U & \text {if } s \notin R(C_T) \text{ and } s \in R(C_U) \\
\f C_T \tsr C_U \\
+[2] \left( \displaystyle \sum_{s \in R(C_{T'})} \mu(T', T) C_{T'} \tsr C_U + \sum_{s \in R(C_{U'})} \mu(U', U) C_{T} \tsr C_{U'} \right) \\
+2 \displaystyle \sum_{s \in R(C_{T'}), s \in R(C_{U'})} \mu(T', T) \mu(U', U) C_{T'} \tsr C_{U'} & \text{if } s \notin R(C_U) \text{ and } s \notin R(C_T) \\
\end{cases}
\end{align*}

\section{Irreducibles of $\nsH_{r,2}$}
\label{s irreducibles of nshr2}
Define the \emph{Temperley-Lieb} algebra $\H_{r,d}$ to be the quotient of $\H_r$ by the two-sided ideal
\[
\bigoplus_{\stackrel{\lambda \vdash r,\ \ell(\lambda) > d,}{P\in \text{SYT}(\lambda)}} \mathbf{A}\Gamma_P = \mathbf{A}\{C_w:\ell(\sh(P(w))) > d\}.
\]
Define the \emph{nonstandard Temperley-Lieb} algebra $\nsH_{r,d}$ to be the subalgebra of $\H_{r,d}\tsrvw \H_{r,d}$ generated by the elements
$\sP_{s} := \C_s \tsrvw \C_s + C_s \tsrvw C_s, \ s \in S$.

Let $\mathscr{P}_{r,2}$ be the set of partitions of size $r$ with at most two parts and $\mathscr{P}'_{r,2}$ be the subset of $\mathscr{P}_{r,2}$ consisting of those partitions that are not a single row or column shape. Define the index set $\nsP_{r,2}$ for the $\field \nsH_{r,2}$-irreducibles as follows:
\be\label{e definition nsp r2}
\nsP_{r,2} =  \{ \{\lambda,\mu\}: \lambda, \mu \in \mathscr{P}_{r,2}, \, \lambda \neq \mu\} \sqcup\{+\lambda: \lambda \in \mathscr{P}'_{r,2}\} \sqcup \{-\lambda: \lambda \in \mathscr{P}'_{r,2}\} \sqcup \{\nsbr{\epsilon}_+\}.
\ee

This section is devoted to a proof of the main result of this paper:
\begin{theorem} \label{t nsH irreducibles two row case}
The algebra $\field\nsH_{r,2}$ is split semisimple and the list of distinct irreducibles is
\begin{list}{\emph{(\arabic{ctr})}} {\usecounter{ctr} \setlength{\itemsep}{1pt} \setlength{\topsep}{2pt}}
\item $\nsbr{M}_\alpha := \nsbr{M}_{\lambda,\mu} = \Res_{\field \nsH_{r,2}}M_\lambda \tsr M_\mu$, for $\alpha = \{\lambda, \mu\} \in \nsP_{r,2}$,
\item $\nsbr{M}_\alpha := S' \nsbr{M}_\lambda$, for $\alpha = +\lambda \in \nsP_{r,2}$,
\item $\nsbr{M}_\alpha := \nswedge{2 }{M}_\lambda$, for $\alpha = -\lambda \in \nsP_{r,2}$,
\item $\nsbr{M}_\alpha := \field\nsbr{\epsilon}_+$, for $\alpha = \nsbr{\epsilon}_+\in \nsP_{r,2}$.
\end{list}
%to be consistent with other notation, these modules should have superscripts \mathbf{A} if we want to state this for integral forms
% seems okay to stated with only the \field forms ??notation
Moreover, the irreducible $\field (\H_{r,2} \tsr \H_{r,2})$-modules decompose into $\field \nsH_{r,2}$-irreducibles as follows
\[
\begin{array}{ll}
M_{\lambda} \tsr M_{\mu} \cong \nsbr{M}_{\lambda, \mu} & \text{if } \lambda \neq \mu, \\
M_{\lambda} \tsr M_{\lambda} \cong S' \nsbr{M}_{\lambda} \oplus \nswedge{2}{M}_{\lambda} \oplus \field \nsbr{\epsilon}_+ & \lambda \vdash r. \\
\end{array}
\]
\end{theorem}

\subsection{Gluing $\field \nsH_{r-1}$-irreducibles} \label{ss gluing field nsH irreducibles}
\begin{proposition} \label{p r-1 restrictions}
The four types of $\field \nsH_{r}$-modules from Theorem \ref{t nsH irreducibles two row case} decompose into $\field \nsH_{r-1}$-modules as follows:
\begin{list}{(\arabic{ctr})} {\usecounter{ctr} \setlength{\itemsep}{1pt} \setlength{\topsep}{2pt}}
\item[\emph{(1a)}] $\Res_{\field\nsH_{r-1}} \nsbr{M}_{\lambda, \mu} \cong \bigoplus_{i \in [k_\lambda], j \in  [k_\mu]} \nsbr{M}_{\lambda - a_i, \mu - b_j}$, if $|\lambda \cap \mu| < r-1$.
\item[\emph{(1b)}] $\Res_{\field\nsH_{r-1}} \nsbr{M}_{\lambda, \mu} \cong \bigoplus_{\substack{i \in [k_\lambda], j \in [k_\mu], \\ (i, j) \neq (k, l)}}  \nsbr{M}_{\lambda - a_i, \mu - b_j} \oplus S' \nsbr{M}_{\nu} \oplus \nswedge{2}{M}_\nu \oplus \field \nsbr{\epsilon}_+$, where  $\nu = \lambda - a_k = \mu - b_l$.
\item[\emph{(2)}] $\Res_{\field\nsH_{r-1}} S' \nsbr{M}_{\lambda} \cong \bigoplus_{1 \leq i < j \leq k_\lambda} \nsbr{M}_{\lambda - a_i, \lambda - a_j} \oplus \bigoplus_{i \in [k_\lambda]} S' \nsbr{M}_{\lambda - a_i} \oplus \field \nsbr{\epsilon}_+^{\oplus k_\lambda -1}$, for $\lambda \in \mathscr{P}'_r$.
\item[\emph{(3)}] $\Res_{\field\nsH_{r-1}} \nswedge{2}{M}_{\lambda} \cong \bigoplus_{1\leq i < j \leq k_\lambda} \nsbr{M}_{\lambda - a_i, \lambda - a_j} \oplus \bigoplus_{i \in [k_\lambda]} \nswedge{2}{M}_{\lambda - a_i}$, for $\lambda \in \mathscr{P}'_r$.
\item[\emph{(4)}] $\Res_{\field\nsH_{r-1}} \field \nsbr{\epsilon}_+ \cong \field \nsbr{\epsilon}_+$.
\end{list}
\end{proposition}
Note that if $|\text{SYT}(\nu)| = 1$, then $S' \nsbr{M}_\nu = \nswedge{2}{M}_\nu = 0$, so some zero modules appear in the right-hand sides. % (these can be ignored).
\begin{proof}
It is well known that $\Res_{\H_{r-1}} M_\lambda \cong \bigoplus_{i \in [k_\lambda]} M_{\lambda-a_i}$.  The decompositions (1a) and (1b) are clear from this and
\eqref{e nsH S' iso}.  Decomposition (3) follows from the general fact that $\Wedge(\bigoplus_{i \in [k]} M_i) \cong \Wedge(M_1)\tsr \ldots \tsr\Wedge(M_k)$ is a graded isomorphism of algebras for any vector spaces  $M_1,\ldots,M_k$, where $\Wedge(M)$ is the exterior algebra of  $M$.
%---in this case, the direct sum is $\Res_{\H_{r-1}} M_\lambda \cong \bigoplus_{i \in [k_\lambda]} M_{\lambda-a_i}$.
The analogous fact holds for symmetric algebras, which implies
\[\Res_{\field S^2\H_{r-1}} S^2 M_{\lambda} \cong \bigoplus_{1 \leq i < j \leq k_\lambda} \nsbr{M}_{\lambda - a_i, \lambda - a_j} \oplus \bigoplus_{i \in [k_\lambda]} S^2 M_{\lambda - a_i}.\]
Decomposition (2) then follows from \eqref{e nsH S' iso}.
\end{proof}

We adopt the convention that restrictions from $\nsH_r$ to  $\nsH_{r-1}$ are considered with respect to the subalgebra of  $\nsH_r$ generated by $\sP_s$,  $s \in J$, where $J := \{s_1,\ldots,s_{r-2}\}$.
Given a  $\field \nsH_{r,d}$-irreducible $N$ and a $\field\nsH_{r,d}$-module $M$, let  $\proj_{N} : M \to M$ denote the  $\field \nsH_{r,d}$ projector with image the $N$-isotypic component of $M$. Given a  $\field \nsH_{r-1,d}$-irreducible  $N$ and a $\field\nsH_r$-module $M$, let  $\projres_{N} : M \to M$ denote the  $\field \nsH_{r-1,d}$ projector with image the $N$-isotypic component of  $\Res_{\field\nsH_{r-1,d}}M$.

%As always, we take the  $\{s_1,\ldots,s_{r-2}\}$-parabolic version of $\nsH_{r-1}$

Theorem \ref{t nsH irreducibles two row case} will be proved inductively, using the list of  $\field \nsH_{r-1,2}$-irreducibles and the fact that the restriction of a $\field \nsH_{r,2}$-irreducible to $\field \nsH_{r-1,2}$ is multiplicity-free.
Let  $\bigoplus_{i \in [k]} \nsbr{M}_i$ be a multiplicity-free decomposition of a  $ \field \nsH_{r}$-module $\nsbr{M}$ into  $\field \nsH_{r-1}$-irreducibles.
Then any  $\field \nsH_{r}$-submodule of  $ \nsbr{ M}$ is a direct sum of some of the  $\nsbr{M}_i$.  Suppose that  $ \nsbr{M}' = \bigoplus_{i \in I} \nsbr{M}_{i}$,  for some $I \subseteq [k]$, is
contained in a $ \field \nsH_{r}$-submodule of  $\nsbr{M}$.  We say that  $\nsbr{M}'$ \emph{glues} to  $\nsbr{M}_j$,  $j \notin I$, if  $\nsbr{M}_j \subseteq \nsbr{M}' (\field\nsH_r)$; this is equivalent to  $\projres_{\nsbr{M}_j}(x) \neq 0$ for some  $x \in \nsbr{M}' (\field\nsH_r)$.  Thus if we show that $\nsbr{M}_1$ glues to  $\nsbr{M}_2$,  $\nsbr{M}_1 \oplus \nsbr{M}_2$ glues to  $\nsbr{M}_3$, $\ldots$, $\bigoplus_{i \in [k-1]} \nsbr{M}_i$ glues to $\nsbr{M}_k$, then  this proves that $ \nsbr{M}$ is a $\field \nsH_{r}$-irreducible.  Slight variants of this argument will be used in the propositions in the next subsection.

\subsection{Four propositions on the irreducibility of $\field \nsH_r$-modules} \label{ss four propositions}
In this subsection we state and prove Propositions \ref{p case lambda mu generic}, \ref{p case lambda one from mu}, \ref{p case wedge lambda}, and \ref{p case sym lambda}, which will be used inductively to show that the $\field \nsH_{r,2}$-modules in (1)--(4) of Theorem \ref{t nsH irreducibles two row case} are irreducible.

\begin{proposition} \label{p case lambda mu generic}
Maintain the setup of  \textsection\ref{ss gluing field nsH irreducibles}. If $\lambda \neq \mu$ and  $\nsbr{M}_{\lambda-a_i, \mu-b_j}$ are distinct irreducible  $\field \nsH_{r-1}$-modules ($i \in [k_\lambda], j \in [k_\mu]$), then  $\nsbr{M}_{\lambda, \mu}$ is an irreducible $\field \nsH_r$-module.
\end{proposition}

\begin{proof}
We work with the basis  $\Gamma'_\lambda \tsr \Gamma'_\lambda$ of  $\nsbr{M}_{\lambda, \mu}$.

It suffices to show that $\nsbr{M}_{\lambda - a_1, \mu - b_1}$ glues to $\nsbr{M}_{\lambda - a_i, \mu - b_j}$ for  $(i,j) \neq (1, 1)$, which we do as follows:
choose $T \in \text{SYT}(\lambda)$ and $U \in \text{SYT}(\mu)$ so that
\be
\parbox{13.5cm}{
\begin{list} {(\arabic{ctr})} {\usecounter{ctr} \setlength{\itemsep}{1pt} \setlength{\topsep}{2pt}}
\item $T_{a_1} = r$, and if $i \neq 1$ then there is an edge $T \dkt{r-1} T'$ with $T'_{a_i} = r$.
\item $U_{b_1} = r$, and if  $j \neq 1$ then there is an edge $U \dkt{r-1} U'$ with $U'_{b_j} = r$.
\end{list}
}
\ee
Such tableaux exist by \eqref{e DKE complete graph}.
Then if  $i \neq 1$ and  $j \neq 1$, then  $s_{r-1} \notin R(\C_T)$,  $s_{r-1} \notin R(\C_{U})$ and  $\C_T \tsr \C_U \sP_{r-1}$ is computed using the last case of (\ref{e sP on C' C'}).  The term $\C_{T'} \tsr \C_{U'}$ appears in the sum.  The projection lemma (Lemma \ref{l projections are not too tricky}) and the fact that
\[
\{(\liftCp_{A})^J \tsr (\liftCp_{B})^J:A \in \text{SYT}(\lambda),\ {A_{a_i}} = r, \ B \in \text{SYT}(\mu), \ B_{b_j} = r\} \cong \Gamma'_{\lambda-a_i}\tsr\Gamma'_{\mu-b_j}
\]
is a $\field_0$-basis of $\L_{\lambda-a_i}\tsr\L_{\mu-b_j}$
shows that the projection of $\C_T \tsr \C_U \sP_{r-1}$ onto $\nsbr{M}_{\lambda-a_i, \mu - b_j}$ is nonzero (the hypotheses of the lemma are satisfied, which depends in a somewhat delicate way on the form of the last case of (\ref{e sP on C' C'})).
Since $\C_T \tsr \C_U \in \nsbr{M}_{\lambda - a_1, \mu - b_1}$ by Corollary \ref{c geck relative a invariant}, $\nsbr{M}_{\lambda - a_1, \mu - b_1}$ glues to $\nsbr{M}_{\lambda-a_i, \mu - b_j}$.  If  $i =1$ or $j=1$, these also glue by the same argument, possibly using the second or third case of (\ref{e sP on C' C'}) instead of the fourth.
\end{proof}

Given a vector space  $M$, let  $\tau : M \tsr M \to M \tsr M$ denote the flip  $a \tsr b \mapsto b \tsr a$.
For $a,b \in M$, put $a \cdot b = \frac{1}{2} (1 + \tau) (a\tsr b) = \frac{1}{2} (a \tsr b +  b \tsr a)$
and $a \wedge b = \frac{1}{2} (1 - \tau) (a\tsr b) = \frac{1}{2} (a \tsr b - b \tsr a).$

Let  $\L_\nu = \field_0 \Gamma'_\nu = \field_0 \Gamma_\nu$ be as defined after Theorem \ref{t transition C' to C}.  Let  $\leq$ be a total order on $\text{SYT}(\lambda)$.    Then
\[S^2 \Gamma'_\nu := \Gamma'_\nu \cdot \Gamma'_\nu = \{\C_{A} \cdot \C_{B} : A,B \in \text{SYT}(\nu),\ A \leq B \}\]
is a basis of $S^2 M_\nu$. Let $S^2 \L_\nu := \field_0 S^2 \Gamma'_\nu$ be the corresponding $\field_0$-lattice of $S^2 M_\nu$.

Similarly,
\[\Wedge^2 \Gamma'_\nu := \{\C_{A} \wedge \C_{B} : A,B \in \text{SYT}(\nu), \ A < B\}\]
is a basis of $\Wedge^2 M_\nu$. Let $\Wedge^2 \L_\nu := \field_0 \Wedge^2 \Gamma'_\nu$ be the corresponding $\field_0$-lattice of $\Wedge^2 M_\nu$.

\begin{lemma}\label{l not in triv}
Fix some $T \in \text{SYT}(\nu)$. The set
\[
\{\proj_{S' \nsbr{M}_\nu} (\C_A \cdot \C_B) : A,B \in \text{SYT}(\nu),\  A < B\} \sqcup \{\proj_{S' \nsbr{M}_\nu} (\C_A \cdot \C_A) : A \in \text{SYT}(\nu), \ A \neq T\}
\]
is a basis of $S' \nsbr{M}_\nu$.
\end{lemma}
\begin{proof}
By Proposition \ref{p triv in Mdual M} (i), $\field \nsbr{\epsilon}_+ \subseteq S^2 M_\nu \subseteq M_\nu \tsr M_\nu$ is spanned by
\be
\label{e not in triv2}
\sum_{Q \in \text{SYT}(\nu)} C_Q \tsr \C_Q \equiv \sum_{Q \in \text{SYT}(\nu)}\C_Q \tsr \C_Q \mod \u S^2\L_{\nu},
\ee
%and since the submodule $\field \nsbr{\epsilon}_+ \subseteq M_\nu \tsr M_\nu$ lies in $S^2 M_\nu$ (by \eqref{??} and
where the equivalence is by Theorem \ref{t transition C' to C}.

As  $S^2 \Gamma'_\nu$ is a basis of  $S^2 M_\nu$, to prove the lemma, it suffices to show that  the left-hand side of \eqref{e not in triv2} is not in the span of $S^2 \Gamma'_\nu \setminus \{\C_T \cdot \C_T\}$.  And this is true because the image of $\sum_{Q \in \text{SYT}(\nu)}\C_Q \tsr \C_Q$ in $S^2\L_{\nu} / \u S^2\L_{\nu}$ is not in the span of the image of $S^2 \Gamma'_\nu \setminus \{\C_T \cdot \C_T\}$ in $S^2\L_{\nu} / \u S^2\L_{\nu}$.
\end{proof}

\begin{proposition} \label{p case lambda one from mu}
Maintain the setup of  \textsection\ref{ss gluing field nsH irreducibles}  and set $\nu = \lambda - a_k = \mu - b_l$.  If the decomposition
\[\Res_{\field\nsH_{r-1}} \nsbr{M}_{\lambda, \mu} \cong \bigoplus_{\substack{i \in [k_\lambda], j \in [k_\mu], \\ (i, j) \neq (k, l)}}  \nsbr{M}_{\lambda - a_i, \mu - b_j} \oplus S' \nsbr{M}_{\nu} \oplus \nswedge{2}{M}_\nu \oplus \field \nsbr{\epsilon}_+\]
of Proposition \ref{p r-1 restrictions} (1b) consists of distinct irreducible  $\field \nsH_{r-1}$-modules,
%\begin{list} {\emph{(\roman{ctr})}} {\usecounter{ctr} \setlength{\itemsep}{1pt} \setlength{\topsep}{2pt}}
%\item $\lambda \gdneq \mu$, $\nu :=\lambda-a_k= \mu - b_l$, ($k \in [k_\lambda], l \in [k_\mu]$)
%\item $\nsbr{M}_{\lambda-a_i, \mu-b_j}$, $(i,j) \in [k_ \lambda] \times [k_\mu] \setminus \{(k,l)\}$, are distinct irreducible  $\field \nsH_{r-1}$-modules,
%\item $S' \nsbr{M}_{\nu}$, $\nswedge{2}{M}_{\nu},$ and $\field \nsbr{\epsilon}_+$ are distinct irreducible $\field \nsH_{r-1}$-modules and are distinct from those in (ii).
%\end{list}
then $\nsbr{M}_{\lambda, \mu}$ is an irreducible $\field \nsH_r$-module.
\end{proposition}
\begin{proof}
First, if  $k_\lambda = k_\mu = 1,$ then  $\lambda = (2)$ and $\mu  = (1,1)$, and the result is clear in this case.  We will then assume $(k,l) \neq (1, 1)$ and prove the proposition using the basis $\Gamma'_\lambda \tsr \Gamma'_\lambda$; if $(k,l) = (1,1)$, the proposition can be proved in a similar way\footnote{The main change required is that \eqref{e sP on C C} must be used in place of \eqref{e sP on C' C'}; these differ by some signs which end up being harmless.} using the argument below with $a_{k_\lambda}, b_{k_\mu}$ in place of  $a_1, b_1$ and the basis $\Gamma_\lambda \tsr \Gamma_\lambda$ in place of  $\Gamma'_\lambda \tsr \Gamma'_\lambda$.

The $\field \nsH_{r-1}$-irreducible $\nsbr{M}_{\lambda - a_1, \mu - b_1}$ glues to
$\nsbr{M}_{\lambda - a_i, \mu - b_j}$ for  $(i,j) \notin \{(k,l),(1, 1)\}$ by the same argument as in the proof of Proposition \ref{p case lambda mu generic}.

The assumption $\lambda \gdneq \mu$ implies  $k \geq l$.  Thus $k >1$ since we are assuming $(k,l) \neq (1,1)$.   We will next show that  $\bigoplus_{i \leq l} \nsbr{M}_{\lambda - a_1, \mu - b_i}$ glues to  $S' \nsbr{M}_{\nu}$ and $\nswedge{2}{M}_{\nu}$.  We may assume that $|\text{SYT}(\nu)| > 1$ because this is equivalent to  $S' \nsbr{M}_{\nu}$ and $\nswedge{2}{M}_{\nu}$ being nonzero.
Thus by \eqref{e DKE complete graph}, we can choose $T, T' \in \text{SYT}(\lambda)$ and $U \in \text{SYT}(\mu)$ such that
\begin{list} {(\arabic{ctr})} {\usecounter{ctr} \setlength{\itemsep}{1pt} \setlength{\topsep}{2pt}}
\item $T_{a_1} = r$ and there is an edge $T \dkt{r-1} T'$ with $T'_{a_k} = r$.
\item $U_{b_l} = r$ and $U_\nu \neq T'_\nu$.
\end{list}
Here  $U_\nu$ denotes the subtableau of $U$ obtained by restricting $U$ to $\nu$.
The quantity $\C_T \tsr \C_U \sP_{r-1}$ is computed using the second or fourth case of \eqref{e sP on C' C'}:
if the second case applies, then the projection lemma shows that
\be
\label{e nu nu case 2}
\projres_{M_\nu \tsr M_\nu}  (\C_T \tsr \C_U \textstyle \frac{\sP_{r-1}}{[2]}) \equiv  \displaystyle \sum_{\substack{s_{r-1} \in R(\C_{A}), \\ A_{a_k} = r}} \mu(A, T)  (\liftCp_{A})^J \tsr (\liftCp_{U})^J \mod  \u \L_\nu \tsr \L_\nu;
\ee
if the fourth case applies, then a careful application of the projection lemma shows that
%note that projection of \C_T onto  M_\nu is 0 (that's why the application requires care)
\be
\label{e nu nu case 4}
\projres_{M_\nu \tsr M_\nu}  (\C_T \tsr \C_U \textstyle \frac{\sP_{r-1}}{[2]}) \equiv  -\displaystyle \sum_{\substack{s_{r-1} \in R(\C_{A}), \\ A_{a_k} = r}} \mu(A, T)  (\liftCp_{A})^J \tsr (\liftCp_{U})^J \mod  \u \L_\nu \tsr \L_\nu.
\ee
Let $x$ (resp.  $-x$) denote the right-hand side of \eqref{e nu nu case 2} (resp.  \eqref{e nu nu case 4}).
Since $\pm(\liftCp_{T'})^J  \tsr (\liftCp_{U})^J$ appears in the expression for $\pm x$ and $T'_\nu \neq U_\nu$, it follows that the projection of $\pm x$ to $\Wedge^2 \L_\nu$ is nonzero.  This uses that
\[
\{(\liftCp_{A})^J \wedge (\liftCp_{B})^J:A \in \text{SYT}(\lambda),\ {A_{a_k}} = r, \ B \in \text{SYT}(\mu), \ B_{b_l} = r, \ A_\nu < B_\nu \} \cong \Wedge^2 \Gamma'_{\nu}
\]
is a $\field_0$-basis of $\Wedge^2 \L_{\nu}$.
The quantities $\pm x$ also have nonzero projection onto $S' \nsbr{M}_{\nu}$ by Lemma \ref{l not in triv}.
Finally, we need that $\C_T \tsr \C_U \in \bigoplus_{i \leq l} \nsbr{M}_{\lambda - a_1, \mu - b_i}$, which holds by Corollary \ref{c geck relative a invariant}, to conclude that $\bigoplus_{i \leq l} \nsbr{M}_{\lambda - a_1, \mu - b_i}$ glues to  $S' \nsbr{M}_{\nu}$ and $\nswedge{2}{M}_{\nu}$.

It remains to show that $\field \nsbr{\epsilon}_+ \subseteq \nsbr{M}_{\lambda - a_k, \mu - b_l}$ glues to some other  $\field \nsH_{r-1}$-irreducible of $\Res_{\field \nsH_{r-1}} \nsbr{M}_{\lambda,\mu}$. If not, then it follows that $\nsbr{\epsilon}_+|_{\u = 1}$ is a 1-dimensional $\QQ \S_r$-submodule of $\nsbr{M}_{\lambda,\mu}|_{\u = 1} \cong M_\lambda|_{\u=1} \tsr M_\mu|_{\u=1}$.  Here, the specialization $N|_{\u=1}$ of an  $\mathbf{A}$-module $N_\mathbf{A}$ is defined to be $\QQ \tsr_{\mathbf{A}} N_\mathbf{A}$, the map $\mathbf{A} \to \QQ$ given by $\u \mapsto 1$.
% ?? This is not completely trivial, as we need to show that a one-dimensional \field \nsbr{\epsilon}_+ submodule yields a one-dimensional submodule for \nsH_r spanned by \nsbr{\epsilon}_+
We are assuming $r \geq 3$, so $\nsbr{\epsilon}_+ \sP_1 = [2]^2 \nsbr{\epsilon}_+$. But then $\nsbr{\epsilon}_+|_{\u =1}$ is the trivial $\QQ \S_r$-module, which is impossible since $\lambda \neq \mu$.
\end{proof}
%?? should we use $\field_0$ or $\field_\infty$?
%I think either is okay.

%Often we have an element $m$ of an  $\H_r \tsr \H_r$-module  $M$ and we wish to decompose  $\Res_{\field \nsH_{r-1}} M$ into irreducibles and project $m$ onto each of these irreducibles.
%We have the canonical projections
%\[
%\begin{array}{rl}
%\proj_{S' \nsbr{M}_\lambda} : \nsbr{M}_{\lambda, \lambda} \to & S' \nsbr{M}_\lambda \\
%\projres_{\nsbr{M}_{\lambda - a_i, \lambda - a_j}} : S' \nsbr{M}_\lambda \to & \nsbr{M}_{\lambda - a_i, \lambda - a_j} \\
%p^1_{M_{\lambda - a_i} \tsr M_{\lambda - a_j}} : M_\lambda \tsr M_\lambda \to & M_{\lambda - a_i} \tsr M_{\lambda - a_j} \\
%\end{array}
%\]

For any  $\field (\H_r \tsr \H_r)$ module  $M$, let $p^1_{M_{\lambda - a_i} \tsr M_{\lambda - a_j}} : M \to M$ be the $\field (\H_{r-1} \tsr \H_{r-1})$ projector with image the $M_{\lambda - a_i} \tsr M_{\lambda - a_j}$-isotypic component of  $M$. For any  $\field \nsH_{r}$-module  $\nsbr{M}$ and $h \in \nsH_r$, let $m_h : \nsbr{M} \to \nsbr{M}$ denote right multiplication by $h$.
\begin{lemma}
\label{l restriction facts}
Let $i, j \in [k_\lambda]$,  $i \neq j$.  There are the following equalities of $\field \nsH_{r-1}$-module endomorphisms of $M_\lambda \tsr M_\lambda$.
\[
\begin{array}{lrcl}
\emph{(i)} & \projres_{\nsbr{M}_{\lambda-a_i,\lambda - a_j}} \textstyle \frac{1-\tau}{2} & = & \textstyle \frac{1-\tau}{2}(p^1_{M_{\lambda-a_i} \tsr M_{\lambda - a_j}}+p^1_{M_{\lambda-a_j} \tsr M_{\lambda - a_i}})\\[2mm]
\emph{(ii)} & \projres_{\nsbr{M}_{\lambda - a_i, \lambda - a_j}} \proj_{S' \nsbr{M}_\lambda} & = & \textstyle \frac{1+\tau}{2} (p^1_{M_{\lambda - a_i} \tsr M_{\lambda - a_j}} +  p^1_{M_{\lambda - a_j} \tsr M_{\lambda - a_i}}) \\[2mm]
\emph{(iii)} & \projres_{\nsbr{M}_{\lambda - a_i, \lambda - a_j}} \proj_{S' \nsbr{M}_\lambda} m_{\sP_{r-1}} \proj_{S' \nsbr{M}_\lambda} & = &
\textstyle \frac{1+\tau}{2}( p^1_{M_{\lambda - a_i} \tsr M_{\lambda - a_j}} +   p^1_{M_{\lambda - a_j} \tsr M_{\lambda - a_i}}) m_{\sP_{r-1}}. \\
\end{array}
\]
\end{lemma}
\begin{proof}
First note that for any  $\H_r \tsr \H_r$-module  $M$, there holds
    \[ \Res_{\nsH_{r-1}} \Res_{\H_{r-1} \tsr \H_{r-1}} M = \Res_{\nsH_{r-1}} M = \Res_{\nsH_{r-1}} \Res_{\nsH_r} M.\]
Statement (i) is immediate from the easy facts
\begin{align*}
\projres_{\nsbr{M}_{\lambda-a_i,\lambda - a_j}} &= p^1_{M_{\lambda-a_i} \tsr M_{\lambda - a_j}}+p^1_{M_{\lambda-a_j} \tsr M_{\lambda - a_i}}, \\
\tau p^1_{M_{\lambda-a_i} \tsr M_{\lambda - a_j}} &= p^1_{M_{\lambda-a_j} \tsr M_{\lambda - a_i}} \tau.
\end{align*}
This also shows that (i) holds with $1+\tau$ in place of  $1-\tau$.  Then
\[
\projres_{\nsbr{M}_{\lambda - a_i, \lambda - a_j}} \textstyle \frac{1+\tau}{2} = \projres_{\nsbr{M}_{\lambda - a_i, \lambda - a_j}} \proj_{\nssym{2}{M}_\lambda} = \projres_{\nsbr{M}_{\lambda - a_i, \lambda - a_j}} (\proj_{S'\nsbr{M}_\lambda}+\proj_{\field \nsbr{\epsilon}_+}) = \projres_{\nsbr{M}_{\lambda - a_i, \lambda - a_j}} \proj_{S' \nsbr{M}_\lambda}
\]
proves (ii).
Statement (iii) is immediate from (ii) and the fact that $\proj_{S' \nsbr{M}_\lambda}$ is a $\field \nsH_r$-module homomorphism.
\end{proof}

We say that the modules in a list are \emph{essentially distinct irreducibles} if the nonzero modules in this list are distinct irreducibles.
\begin{proposition} \label{p case wedge lambda}
Maintain the setup of  \textsection\ref{ss gluing field nsH irreducibles} and assume $\lambda \in \mathscr{P}'_r$.
If $\nswedge{2}{M}_{\lambda - a_i},\ i \in [k_{\lambda}]$, and $\nsbr{M}_{\lambda - a_i, \lambda - a_j},\ i < j,\ i, j \in [k_{\lambda}]$, are essentially distinct irreducible $\field \nsH_{r-1}$-modules, then $\nswedge{2}{M}_{\lambda}$ is an irreducible $\field \nsH_r$-module.
\end{proposition}
\begin{proof}
We work with the basis $\Wedge^2 \Gamma'_\lambda$ of  $\nswedge{2}{M}_\lambda$.

Let  $i > 1$ and assume $\nswedge{2}{M}_{\lambda - a_1}$ is nonzero.  We show that $\nswedge{2}{M}_{\lambda - a_1}$ glues to $\nsbr{M}_{\lambda - a_i, \lambda -a_1}$ as follows:
given the assumptions, we can choose $T, U \in \text{SYT}(\lambda)$ so that
\begin{list} {(\arabic{ctr})} {\usecounter{ctr} \setlength{\itemsep}{1pt} \setlength{\topsep}{2pt}}
\item $T_{a_1} = r$ and there is an edge $T \dkt{r-1} T'$ with $T'_{a_i} = r$.
\item $U \neq T$ and $U_{a_1} = r$.
\end{list}
If  $s_{r-1} \not\in R(\C_U)$, then $\C_T \tsr \C_U \sP_{r-1}$ is computed using the fourth case of \eqref{e sP on C' C'}.
Lemma \ref{l restriction facts} (i) yields the first equality and the projection lemma yields the equivalence in the following
\begin{align*}
 \projres_{\nsbr{M}_{\lambda-a_i,\lambda - a_1}} (\C_T \wedge \C_U \textstyle \frac{\sP_{r-1}}{[2]})  = \textstyle \frac{1-\tau}{2}(p^1_{M_{\lambda-a_i} \tsr M_{\lambda - a_1}}+p^1_{M_{\lambda-a_1} \tsr M_{\lambda - a_i}}) (\C_T \tsr \C_U \textstyle \frac{\sP_{r-1}}{[2]})\\
  \equiv -\displaystyle \sum_{\substack{s_{r-1} \in R(\C_{A}), \\ A_{a_i} = r}} \mu(A, T)  (\liftCp_{A})^J \wedge (\liftCp_{U})^J
   - \displaystyle \sum_{\substack{s_{r-1} \in R(\C_{B}), \\ B_{a_i} = r}} \mu(B, U)  (\liftCp_{T})^J \wedge (\liftCp_{B})^J \\
   = -\displaystyle \sum_{\substack{s_{r-1} \in R(\C_{A}), \\ A_{a_i} = r}} \mu(A, T)  (\liftCp_{A})^J \wedge (\liftCp_{U})^J
+\displaystyle \sum_{\substack{s_{r-1} \in R(\C_{B}), \\ B_{a_i} = r}} \mu(B, U)  (\liftCp_{B})^J \wedge (\liftCp_{T})^J.
\end{align*}
The equivalence is mod $\u (\L_{\lambda-a_i} \tsr \L_{\lambda-a_1} \oplus \L_{\lambda-a_1} \tsr \L_{\lambda-a_i})$.
The final line is nonzero because $(\liftCp_{T'})^J \wedge (\liftCp_{U})^J$ appears in the left sum,  $U \neq T$, and $\Gamma'_{\lambda-a_i} \tsr \Gamma'_{\lambda-a_1}$ is a $\field_0$-basis of  $\L_{\lambda-a_i} \tsr \L_{\lambda-a_1} \subseteq \nsbr{M}_{\lambda-a_i,\lambda - a_1}$.\footnote{Throughout this proof $\nsbr{M}_{\lambda-a_i,\lambda - a_1}$ is understood as a $\field \nsH_{r-1}$-submodule of  $\Res_{\field \nsH_{r-1}} \nswedge{2}{M}_\lambda$.}
A similar (but easier) argument shows that  $ \projres_{\nsbr{M}_{\lambda-a_i,\lambda - a_1}} (\C_T \wedge \C_U \frac{\sP_{r-1}}{[2]})$ is nonzero in the case $s_{r-1} \in R(\C_U)$.
Thus since $\C_T \wedge \C_U \in \nswedge{2}{M}_{\lambda - a_1}$ by Corollary \ref{c geck relative a invariant},  $\nswedge{2}{M}_{\lambda-a_1 }$ glues to $\nsbr{M}_{\lambda-a_i,\lambda - a_1}$  ($i > 1$).

We next show that\footnote{By definition, $\nsbr{M}_{\lambda-a_1, \lambda-a_2}= \nsbr{M}_{\lambda-a_2, \lambda-a_1}$; we work with the former here to keep notation more consistent with other parts of the proof of Theorem \ref{t nsH irreducibles two row case}.}
$\nsbr{M}_{\lambda-a_1, \lambda-a_2} \oplus \nswedge{2}{M}_{\lambda-a_1}$
glues to  $\nswedge{2}{M}_{\lambda-a_2}$.
Since we can assume $\nswedge{2}{M}_{\lambda-a_2}$ is nonzero, we can choose $T, U \in \text{SYT}(\lambda)$ so that
\begin{list} {(\arabic{ctr})} {\usecounter{ctr} \setlength{\itemsep}{1pt} \setlength{\topsep}{2pt}}
\item $T_{a_1} = r$ and there is an edge $T \dkt{r-1} T'$ with $T'_{a_2} = r$.
\item $U_{a_2} = r$, and $U \neq T'$.
\end{list}
If  $s_{r-1} \not\in R(\C_U)$, then $\C_T \tsr \C_U \sP_{r-1}$ is computed using the fourth case of \eqref{e sP on C' C'}.
A careful application of the projection lemma shows that
\begin{align*}
\projres_{\nswedge{2}{M}_{\lambda-a_2}} (\C_T \wedge \C_U \textstyle\frac{\sP_{r-1}}{[2]}) &= \textstyle \frac{1-\tau}{2} \projres_{M_{\lambda-a_2} \tsr M_{\lambda - a_2}} (\C_T \tsr \C_U \textstyle \frac{\sP_{r-1}}{[2]})\\
&\equiv -\displaystyle \sum_{\substack{s_{r-1} \in R(\C_{A}), \\ A_{a_2} = r}} \mu(A, T)  (\liftCp_{A})^J \wedge (\liftCp_{U})^J
\mod \u \Wedge^2 \L_{\lambda-a_2}
\end{align*}
The last line is nonzero because $(\liftCp_{T'})^J \wedge (\liftCp_{U})^J$ appears in the sum, $U \neq T'$, and $\Wedge^2 \Gamma'_{\lambda-a_2}$ is a  $\field_0$-basis of  $\Wedge^2 \L_{\lambda-a_2}$.
A similar (but easier) argument shows that  $ \projres_{\nswedge{2}{M}_{\lambda-a_2}} (\C_T \wedge \C_U \textstyle \frac{\sP_{r-1}}{[2]})$ is nonzero in the case $s_{r-1} \in R(\C_U)$.
Thus since $\C_T \wedge \C_U \in \nsbr{M}_{\lambda-a_1, \lambda-a_2} \oplus \nswedge{2}{M}_{\lambda-a_1}$ by Corollary \ref{c geck relative a invariant},  $\nsbr{M}_{\lambda-a_1, \lambda-a_2} \oplus \nswedge{2}{M}_{\lambda-a_1}$ glues to  $\nswedge{2}{M}_{\lambda-a_2}$.
Note that this argument still works if $\nswedge{2}{M}_{\lambda-a_1} = 0$.

Repeating the arguments of the previous two paragraphs, one shows that $\bigoplus_{1 < i \leq k_\lambda} \nsbr{M}_{\lambda-a_1, \lambda-a_i} \oplus \bigoplus_{i \in \{1,2\}} \nswedge{2}{M}_{\lambda-a_i}$ glues to $\nsbr{M}_{\lambda-a_2, \lambda-a_j}$ for  $j > 2$, $\bigoplus_{\substack{1 \leq i < j \leq k_\lambda, \\ i \leq 2}} \nsbr{M}_{\lambda-a_i, \lambda-a_j} \oplus \bigoplus_{i \in \{1,2\}} \nswedge{2}{M}_{\lambda-a_i}$ glues to $\nswedge{2}{M}_{\lambda-a_3}$, etc., which shows
that all the irreducible constituents of $\Res_{\field \nsH_{r-1}} \nswedge{2}{M}_{\lambda}$ are contained in a single $\field \nsH_r$-irreducible.
\end{proof}

\begin{proposition} \label{p case sym lambda}
Maintain the setup of  \textsection\ref{ss gluing field nsH irreducibles} and assume $\lambda \in \mathscr{P}'_r$.
If $S' \nsbr{M}_{\lambda - a_i},\ i \in [k_{\lambda}]$ and $\nsbr{M}_{\lambda - a_i, \lambda - a_j},\ i < j,\ i, j \in [k_{\lambda}]$, are essentially distinct irreducible $\field \nsH_{r-1}$-modules, then $S'\nsbr{M}_{\lambda}$ is an irreducible $\field \nsH_r$-module.
\end{proposition}
Note that $S'\nsbr{M}_{\lambda}$ does not necessarily have a multiplicity-free decomposition into $\field \nsH_{r-1}$-irreducibles, but the proof method explained in \textsection\ref{ss gluing field nsH irreducibles} still gives most of the proof.  The  $\field \nsH_{r-1}$-irreducible  $\field \nsbr{\epsilon}_+$ may appear with multiplicity more than one, so it is handled separately.
\begin{proof}
We work with the basis $\Gamma_\lambda \tsr \Gamma'_\lambda$ of  $M_\lambda \tsr M_\lambda$.

If  $k_\lambda =1$, then  $\Res_{\field \nsH_{r-1}} S'\nsbr{M}_\lambda \cong S'\nsbr{M}_{\lambda-a_1}$, so the result holds.
Assume $k_\lambda >1$.
First we show that $\nsbr{M}_{\lambda - a_{k_\lambda}, \lambda-a_1}$ glues to $\nsbr{M}_{\lambda - a_i, \lambda - a_j}$ for $i > j$,  $(i,j) \neq (k_\lambda,1)$, as follows: choose $T, U \in \text{SYT}(\lambda)$ so that
\begin{list} {(\arabic{ctr})} {\usecounter{ctr} \setlength{\itemsep}{1pt} \setlength{\topsep}{2pt}}
\item $T_{a_{k_\lambda}} = r$, and if $i \neq k_\lambda$ then there is an edge $T \dkt{r-1} T'$ with $T'_{a_i} = r$.
\item $U_{a_1} = r$, and if  $j \neq 1$ then there is an edge $U \dkt{r-1} U'$ with $U'_{a_j} = r$.
\end{list}
Put $x = C_T \tsr \C_U$. We wish to show that
\be \label{e restrict two orders}
\projres_{\nsbr{M}_{\lambda - a_i, \lambda - a_j}} \proj_{S' \nsbr{M}_\lambda} m_{\sP_{r-1}} \proj_{S' \nsbr{M}_\lambda} x =
\textstyle \frac{1+\tau}{2}(p^1_{M_{\lambda - a_i} \tsr M_{\lambda - a_j}} +   p^1_{M_{\lambda - a_j} \tsr M_{\lambda - a_i}}) (x\sP_{r-1})
\ee
is nonzero (the equality is by Lemma \ref{l restriction facts} (iii)).  This is shown in three cases.

\noindent
The case $i \neq k_\lambda$ and $j \neq 1$: $s_{r-1} \notin R(C_T)$ and $s_{r-1} \notin R(\C_{U})$, so  $x \sP_{r-1}$ is computed using the fourth case of (\ref{e sP on C C'}). There holds
\begin{align*}
&\textstyle \frac{1+\tau}{2}(p^1_{M_{\lambda-a_i} \tsr M_{\lambda - a_j}}+p^1_{M_{\lambda-a_j} \tsr M_{\lambda - a_i}}) (x \sP_{r-1})\\
&\equiv \displaystyle \sum_{\substack{s_{r-1} \in R(C_{A}), \\ s_{r-1} \in  R(\C_B),\\ A_{a_i} = r, \ B_{a_j} = r}} \mu(A, T)\mu(B,U)  (\liftCup_{A})^J \cdot (\liftCp_{B})^J + \displaystyle \sum_{\substack{s_{r-1} \in R(C_{A}), \\ s_{r-1} \in  R(\C_B),\\ A_{a_j} = r, \ B_{a_i} = r}} \mu(A, T)\mu(B,U)  (\liftCp_{B})^J \cdot (\liftCup_{A})^J, \\
&\equiv \displaystyle \sum_{\substack{s_{r-1} \in R(C_{A}), \\ s_{r-1} \in  R(\C_B),\\ A_{a_i} = r, \ B_{a_j} = r}} \mu(A, T)\mu(B,U)  (\liftCp_{A})^J \cdot (\liftCp_{B})^J + \displaystyle \sum_{\substack{s_{r-1} \in R(C_{A}), \\ s_{r-1} \in  R(\C_B),\\ A_{a_j} = r, \ B_{a_i} = r}} \mu(A, T)\mu(B,U)  (\liftCp_{B})^J \cdot (\liftCp_{A})^J,
\end{align*}
where the first equivalence is by the projection lemma, the second is by Theorem \ref{t transition C' to C}, and the equivalences are mod $\u \L_{\lambda-a_i} \tsr \L_{\lambda-a_j}$.
The last line is nonzero because $(\liftCp_{T'})^J \cdot (\liftCp_{U'})^J$ appears in the left sum, the coefficients $\mu(A,T)\mu(B,U)$ are nonnegative (Theorem \ref{t positive coefficients}), and $\Gamma'_{\lambda-a_i} \tsr \Gamma'_{\lambda-a_j}$ is a $\field_0$-basis of $\L_{\lambda-a_i} \tsr \L_{\lambda-a_j}$.

\smallskip
\noindent
The case $i \neq k_\lambda, j = 1$ (the  $i = k_\lambda, j \neq 1$ case is similar): $x \sP_{r-1}$ is computed using the third or fourth case of (\ref{e sP on C C'}).
A careful application of the projection lemma yields
\begin{align*}
&\textstyle \frac{1+\tau}{2}(p^1_{M_{\lambda-a_i} \tsr M_{\lambda - a_j}}+p^1_{M_{\lambda-a_j} \tsr M_{\lambda - a_i}}) (x \textstyle \frac{\sP_{r-1}}{[2]})\\
&  \equiv \pm\displaystyle \sum_{\substack{s_{r-1} \in R(C_{A}), \\ A_{a_i} = r}} \mu(A, T)  (\liftCup_{A})^J \cdot (\liftCp_{U})^J
\mod \u \L_{\lambda-a_i} \tsr \L_{\lambda-a_j}
\end{align*}
The second line is nonzero because $(\liftCup_{T'})^J \cdot (\liftCp_{U})^J$ appears in sum and $\Gamma_{\lambda-a_i} \tsr \Gamma'_{\lambda-a_j}$ is a $\field_0$-basis of  $\L_{\lambda-a_i} \tsr \L_{\lambda-a_j}$.

\smallskip
It follows from Proposition \ref{p triv in Mdual M} (ii) that  $\proj_{S' \nsbr{M}_{\lambda}} x = C_T \cdot \C_U$.  Then by Lemma \ref{l restriction facts} (ii) and Corollary \ref{c geck relative a invariant}, $\projres_{\nsbr{M}_{\lambda - a_{k_\lambda}, \lambda - a_1}} \proj_{S' \nsbr{M}_{\lambda}} x = C_T \cdot \C_U$, so  $\proj_{S' \nsbr{M}_{\lambda}} x \in \nsbr{M}_{\lambda - a_{k_\lambda}, \lambda-a_1} \subseteq  S' \nsbr{M}_{\lambda}$.  Hence the left-hand side of \eqref{e restrict two orders} being nonzero implies that $\nsbr{M}_{\lambda - a_{k_\lambda}, \lambda-a_1}$ glues to $\nsbr{M}_{\lambda - a_i, \lambda - a_j}$.

Fix $i \in [k_\lambda-1]$ and set $\nu = \lambda - a_i$.
Now we show that $\bigoplus_{j \leq i} \nsbr{M}_{\lambda - a_{k_\lambda}, \lambda - a_j}$ glues to $S' \nsbr{M}_{\nu}$.
Choose $T,U \in \text{SYT}(\lambda)$ so that
\begin{list} {(\arabic{ctr})} {\usecounter{ctr} \setlength{\itemsep}{1pt} \setlength{\topsep}{2pt}}
\item $T_{a_{k_\lambda}} = r$ and there is an edge $T \dkt{r-1} T'$ with $T'_{a_i} = r$.
\item $U_{a_i} = r$ and $U \neq T'$.
\end{list}
This is possible since we can assume $S' \nsbr{M}_{\nu}$ is nonzero, which is equivalent to $|\text{SYT}(\nu)| > 1$.
Then $C_T \cdot \C_U \sP_{r-1}$ is computed using the third or fourth case of (\ref{e sP on C C'}) with  $\cdot$ in place of  $\tsr$. A careful application of the projection lemma  yields the first equivalence below
\begin{align*}
\projres_{S' \nsbr{M}_{\nu}} p^1_{M_{\nu} \tsr M_{\nu}} \big(C_T \cdot \C_U \textstyle \frac{\sP_{r-1}}{[2]}\big)
&\equiv
\projres_{S' \nsbr{M}_{\nu}} \Big( \pm \displaystyle \sum_{\substack{s_{r-1} \in R(C_A), \\ A_{a_i} = r}} \mu(A,T) (\liftCup_{A})^J \cdot (\liftCp_U)^J \Big)
\\
&\equiv \pm \displaystyle \sum_{\substack{s_{r-1} \in R(C_A), \\ A_{a_i} = r}} \mu(A,T) \projres_{S' \nsbr{M}_{\nu}}\big((\liftCp_{A})^J \cdot (\liftCp_U)^J\big)
\mod \u \proj_{S' \nsbr{M}_\nu} (\L_{\nu}\tsr \L_\nu).
\end{align*}
The second equivalence is by Theorem \ref{t transition C' to C}.
It follows from Lemma \ref{l not in triv} and $U \neq T'$ that the second line is nonzero.
By an argument similar to that in the previous paragraph, $C_T \cdot \C_U = \proj_{S' \nsbr{M}_{\lambda}} (C_T \tsr \C_U) \in \bigoplus_{j \leq i} \nsbr{M}_{\lambda - a_{k_\lambda}, \lambda - a_j}$, hence $\bigoplus_{j \leq i} \nsbr{M}_{\lambda - a_{k_\lambda}, \lambda - a_j}$ glues to $S' \nsbr{M}_{\nu}$.

By an argument similar to the  $i =1$ case of the previous paragraph, $\nsbr{M}_{\lambda - a_{k_\lambda}, \lambda - a_1}$ glues to $S' \nsbr{M}_{\lambda-a_{k_\lambda}}$.

Let  $X_\epsilon \subseteq S'\nsbr{M}_{\lambda}$ be the isotypic component of $\Res_{\field \nsH_{r-1}} S'\nsbr{M}_{\lambda}$ of irreducible type $\field \nsbr{\epsilon}_+$ and let  $X_\epsilon^\mathbf{A} := \bigcap_{i \in [r-2]} \ker(m_{\sQ_i})$ be an integral form of $X_\epsilon$, where
$m_{\sQ_i}: S'\nsbr{M}_{\lambda}^\mathbf{A} \to S'\nsbr{M}_{\lambda}^\mathbf{A}$ is right multiplication by  $\sQ_i$; there holds  $\field \tsr_\mathbf{A} X_\epsilon^\mathbf{A} \cong X_\epsilon.$  To complete the proof, it suffices to show that  $x \H_r \not \subseteq X_\epsilon^\mathbf{A}$ for any  $x\in X_\epsilon^\mathbf{A}$.
If  $x \H_r \subseteq X_\epsilon^\mathbf{A}$, then  $\Res_{\QQ \S_{r-1}} (x\H_r|_{\u=1})$ is a direct sum of copies of the trivial $\QQ \S_{r-1}$-module (where $N|_{\u=1}$ of an  $\mathbf{A}$-module $N_\mathbf{A}$ is defined to be $\QQ \tsr_{\mathbf{A}} N_\mathbf{A}$, the map $\mathbf{A} \to \QQ$ given by $\u \mapsto 1$).  It follows that $x\H_r|_{\u=1}$ is a direct sum of copies of the trivial $\QQ \S_r$-module.  But this is impossible since there are no copies of the trivial  $\QQ \S_r$-module in $S'\nsbr{M}_{\lambda}|_{\u=1}$.
\end{proof}

\subsection{Completing the proof}
\begin{proof}[Proof of Theorem \ref{t nsH irreducibles two row case}]
The proof is by induction on $r$. Given that the $\field \nsH_{r-1,2}$-irreducibles of the theorem are distinct, it follows from Propositions \ref{p case lambda mu generic}, \ref{p case lambda one from mu}, \ref{p case wedge lambda}, and \ref{p case sym lambda} that the $\field \nsH_{r,2}$-modules in (1)--(4) are irreducible.
The list of $\field \nsH_{r,2}$-irreducibles is complete because  $\Res_{\field \nsH_{r,2}} \field (\H_{r,2} \tsr \H_{r,2})$ is a faithful $ \field \nsH_{r,2}$-module and all the  $\field \nsH_{r,2}$-irreducible constituents of  $\nsbr{M}_{\lambda, \mu}$ appear in the list.
Also, the split semisimplicity of $\field \nsH_{r,2}$ follows from the proofs of Propositions \ref{p case lambda mu generic}, \ref{p case lambda one from mu}, \ref{p case wedge lambda}, and \ref{p case sym lambda} since these work just as well over any field extension of $\field$.
We now must show that the irreducibles in the list are distinct.

For this we apply Proposition \ref{p r-1 restrictions} and refine the cases as follows:
 %can use the additional claim about the decomposition of a $\field \nsH_r$-module into $\field \field \nsH_{r-1,2}$-irreducibles.
\begin{list}{} {\usecounter{ctr} \setlength{\itemsep}{1pt} \setlength{\topsep}{2pt}}
\item[(1a)] $\Res_{\field\nsH_{r-1,2}} \nsbr{M}_{\lambda, \mu} \cong \bigoplus_{i \in [k_\lambda], j \in [k_\mu]} \nsbr{M}_{\lambda - a_i, \mu - b_j}$, if $|\lambda \cap \mu| < r-1$.
\item[(1b)] $\Res_{\field\nsH_{r-1,2}} \nsbr{M}_{\lambda, \mu} \cong \bigoplus_{\substack{i \in [k_\lambda], j \in [k_\mu], \\ (i, j) \neq (k, l)}} \nsbr{M}_{\lambda - a_i, \mu - b_j} \oplus S' \nsbr{M}_{\nu} \oplus \nswedge{2}{M}_\nu \oplus \field \nsbr{\epsilon}_+$, where $\nu = \lambda - a_k = \mu - b_l$ and $\nu \neq (r-1)$.
\item[(1b$'$)] $\Res_{\field\nsH_{r-1,2}} \nsbr{M}_{(r), (r-1,1)} \cong \nsbr{M}_{(r-1), (r-2,1)} \oplus \field \nsbr{\epsilon}_+$.
\item[(2)] $\Res_{\field\nsH_{r-1,2}} S' \nsbr{M}_{\lambda} \cong \bigoplus_{1 \leq i < j \leq k_\lambda} \nsbr{M}_{\lambda - a_i, \lambda - a_j} \oplus \bigoplus_{i \in [k_\lambda]} S' \nsbr{M}_{\lambda - a_i} \oplus \field \nsbr{\epsilon}_+^{\oplus k_\lambda -1}$, for $+\lambda \in \nsP_{r,2}$, $\lambda \neq (r-1,1)$.
\item[(2$'$)] $\Res_{\field\nsH_{r-1,2}} S' \nsbr{M}_{(r-1,1)} \cong  \nsbr{M}_{(r-1), (r-2,1)} \oplus S' \nsbr{M}_{(r-2,1)} \oplus \field \nsbr{\epsilon}_+$, $r > 2$.
\item[(3)] $\Res_{\field\nsH_{r-1,2}} \nswedge{2}{M}_{\lambda} \cong \bigoplus_{1 \leq i < j \leq k_\lambda} \nsbr{M}_{\lambda - a_i, \lambda - a_j} \oplus \bigoplus_{i \in [k_\lambda]} \nswedge{2}{M}_{\lambda - a_i}$, for $-\lambda \in \nsP_{r,2}$, \newline $\lambda \neq (r-1,1)$.
\item[(3$'$)] $\Res_{\field\nsH_{r-1,2}} \nswedge{2}{M}_{(r-1,1)} \cong  \nsbr{M}_{(r-1), (r-2,1)} \oplus \nswedge{2}{M}_{(r-2,1)}$, $r > 2$.
\item[(4)] $\Res_{\field\nsH_{r-1,2}} \field \nsbr{\epsilon}_+ \cong \field \nsbr{\epsilon}_+$.
\end{list}
Note that for  $r=3$, $S' \nsbr{M}_{(1,1)}$ and $\nswedge{2}{M}_{(1,1)}$ are zero in the right-hand sides of (2$'$) and (3$'$), respectively.

For $r \leq 3$, we check by hand that all these irreducibles are distinct. In particular, we must check that $\nsbr{M}_{(3), (2, 1)} \not \cong S' \nsbr{M}_{(2,1)}$, which happen to have isomorphic restrictions to $\nsH_{2,2}$.

Assuming that $r > 3$ we will show that this list of irreducibles does not contain repetitions by showing that the irreducibles have distinct restrictions to $\field \nsH_{r-1,2}$. We do this in two steps:
\begin{list} {(\Alph{ctr})} {\usecounter{ctr} \setlength{\itemsep}{1pt} \setlength{\topsep}{2pt}}
\item The $\field \nsH_{r-1,2}$ restrictions of any two irreducibles of a given type above are nonisomorphic.
\item The $\field \nsH_{r-1,2}$ restriction of an irreducible of type ($\alpha$) is not isomorphic to the $\field \nsH_{r-1, 2}$ restriction of an irreducible of type ($\beta$), if $\alpha \neq \beta$.
\end{list}
Claim (A) is straightforward: for example, to see that two irreducibles of type (1a) are distinct, suppose  $M$ is a  $\field \nsH_{r,2}$-module of type (1a) and $\Res_{\field \nsH_{r-1,2}} M \cong \bigoplus_{i \in [l]} \nsbr{M}_{\nu^{(2i-1)}, \nu^{(2i)}}$, for some $\nu^{(j)} \vdash r-1$.  The set of partitions
$\{\nu^{(i)} \cup \nu^{(j)}: i, j \in [2l],\ |\nu^{(i)} \cup \nu^{(j)}| = r\}$ consists of two partitions, call them $\lambda$ and  $\mu$.  Then $M = \nsbr{M}_{\lambda,\mu}$.

For claim (B), we can look at which $\field \nsH_{r-1,2}$ restrictions have no occurrences of an irreducible of the form $S' \nsbr{M}_\nu$ and which ones have at least one occurrence of an irreducible of the form $S' \nsbr{M}_\nu$, and similarly for the forms $\nswedge{2}{M}_\nu$ and $\field \nsbr{\epsilon}_+$.
This yields the claim (B) for all pairs of types except (2) and (2$'$), (3) and (3$'$), and (1b$'$) and (4) which are all easy to check directly.
%(2) and (2$'$): it is easy to check that $\Res_{\field \nsH_{r-1,2}} S' \nsbr{M}_{(r-1,1)} \not \cong \Res_{\field \nsH_{r-1,2}} S' \nsbr{M}_{(r-2,2)}$.
%(3) and (3$'$): it is easy to check that $\Res_{\field \nsH_{r-1,2}} \nswedge{2}{M}_{(r-1,1)} \not \cong \Res_{\field \nsH_{r-1,2}} \nswedge{2}{M}_{(r-2,2)}$.
%(1b$'$) and (4): these are clearly nonisomorphic $\field \nsH_{r-1,2}$ restrictions.

The part of the theorem about the decomposition of $\field (\H_{r,2} \tsr \H_{r,2})$-modules into $\field \nsH_{r,2}$-irreducibles is immediate from the definitions in \textsection\ref{ss some representation theory of nsH}.
\end{proof}

\begin{remark}
It is possible that this proof would be easier using a Hecke algebra analog of Young's orthogonal basis (see \cite{Wenzl}) instead of the lower and upper canonical bases.  However, we believe it to be important to understand the action of $\nsH_{r,2}$ on the lower and upper canonical basis of $ \nsbr{M}_{\lambda,\mu}$ anyway. The canonical bases also have the advantage that all computations except the projections onto  $\field \nsH_{r-1,2}$-irreducible isotypic components take place over $\mathbf{A}[\frac{1}{[2]}]$ rather than $\field$ or some extension of $\field$.
Moreover, once the results of  \textsection\ref{ss lifts} are in place, the only thing we need to know about the  $\S_r$-graphs $\Gamma'_\lambda$ and  $\Gamma_\lambda$ are the edges corresponding to dual Knuth transformations; it is likely that a proof using a Hecke orthogonal basis would amount to showing the existence of certain dual Knuth transformations in a similar way.
%??could use a little rephrasing, but not essential
\end{remark}

\section{Seminormal bases}
\label{s Seminormal bases}
We recall the definition of a seminormal basis from \cite{RamSeminormal}, observe that $\field\nsH_{r,2}$-irreducibles have seminormal bases, and give combinatorial labels for the elements of these bases.

\begin{definition}\label{d seminormal}
Given a chain of semisimple $\field$-algebras $\field\cong H_1 \subseteq H_2 \subseteq \dots \subseteq H_r$ and an $H_r$-module $N_\lambda$, a \emph{seminormal basis} of $N_\lambda$ is a $\field$-basis $B$ of $N_\lambda$ compatible with the restrictions in the following sense: there is a partition $B = B_{\mu^1} \sqcup \dots \sqcup B_{\mu^k}$ such that $N_\lambda \cong N_{\mu^1} \oplus \dots \oplus N_{\mu^k}$ as $H_{r-1}$-modules, where $N_{\mu^i} = \field B_{\mu^i}$. Further, there is a partition of each $B_{\mu^i}$ that gives rise to a decomposition of $N_{\mu^i}$ into $H_{r-2}$-irreducibles, and so on, all the way down to $H_1$.
\end{definition}
If the restriction of an $H_i$-irreducible to $H_{i-1}$ is multiplicity-free for all $i$, then a seminormal basis of an  $H_r$-irreducible is unique up to a diagonal transformation.

A consequence of Theorem \ref{t nsH irreducibles two row case} and Proposition \ref{p r-1 restrictions} is that the restriction of a $ \field \nsH_{r,2}$-irreducible to $\field \nsH_{r-1,2}$ is multiplicity-free.  Thus each $\field \nsH_{r,2}$-irreducible  $\nsbr{M}_\alpha$,  $\alpha \in \nsP_{r,2}$, has a seminormal basis  $\nsbr{\text{SN}}_\alpha$ that is unique up to a diagonal transformation. We adopt the convention to take the seminormal basis with respect to the chain $\field\nsH_{J_1} \subseteq \cdots \subseteq \field\nsH_{J_{r-1}} \subseteq \field\nsH_{J_{r}}$, where  $J_i = \{s_1,\ldots,s_{i-1}\}$ and  $\nsH_L$ (for $L \subseteq S$) is the subalgebra of $\nsH_{r,2}$ generated by  $\sP_s$,  $s \in L$.

For $\lambda, \mu \vdash r$ with $\ell(\lambda), \ell(\mu) \leq 2$, $M_\lambda \tsr M_\mu$ has a multiplicity-free decomposition into $\nsH_{r,2}$-modules (by Theorem \ref{t nsH irreducibles two row case}). Thus we can also define a seminormal basis  $\nsbr{ \text{SN}}_{\lambda,\mu}$ of $M_\lambda \tsr M_\mu$ to be the union of the seminormal bases of its $\field \nsH_{r,2}$-irreducible constituents.

%from GCT4:
%However, we have gained \emph{something}. In addition to the slight help that \eqref{e ns Kronecker decomposition} provides, finding a basis for $\nsbr{M}_\alpha$ whose cells are compatible with the decomposition  $\nsbr{M}_\alpha|_{q=1} \cong \bigoplus_\nu (M_\nu|_{q=1})^{\oplus m_{\alpha\nu}}$ has significantly more structure than finding a basis for $M_\lambda\tsr M_\mu|_{q=1}$ compatible with the decomposition  $M_\lambda\tsr M_\mu|_{q=1} \cong \bigoplus_\nu (M_\nu|_{q=1})^{\oplus g_{\lambda \mu \nu}}$, despite the fact that $\nsbr{M}_\alpha$ is typically equal to some $\Res_{\nsH_r} M_\lambda\tsr M_\mu$; that is, finding a basis \emph{before} specializing  $q=1$ has more structure than finding it \emph{after} specializing.

We are interested in these seminormal bases primarily as a tool for constructing a canonical basis of a $\field \nsH_{r,2}$-irreducible that is compatible with its decomposition into irreducibles at $\u =1$, as described in \cite[\textsection19]{BMSGCT4}.  Even though the irreducibles of  $\field \nsH_{r,2}$ are close to those of  $\field(\H_{r,2}\tsr\H_{r,2})$,
the seminormal basis  $\nsbr{\text{SN}}_{\lambda,\mu}$ of $M_\lambda\tsr M_\mu$ using the chain  $\field\nsH_{J_1} \subseteq \cdots \subseteq \field \nsH_{J_{r-1}} \subseteq \field\nsH_{J_r}$ is significantly different from the seminormal basis using the chain $\field(\H_{1,2} \tsr \H_{1,2}) \subseteq \cdots \subseteq \field (\H_{r-1,2} \tsr \H_{r-1,2}) \subseteq \field (\H_{r,2} \tsr \H_{r,2})$.
Thus even though the representation theory of the nonstandard Hecke algebra alone is not enough to understand Kronecker coefficients,
there is hope that the seminormal bases $\nsbr{\text{SN}}_{\lambda,\mu}$ will yield a better understanding of Kronecker coefficients.

\begin{remark}
The $\field \nsH_6$-module  $\nsbr{M}_{(4,1,1),(3,2,1)}$ is irreducible and its $\field \nsH_5$ restriction is not multiplicity-free.  However, we suspect that $\field \nsH_{r-1}$ restrictions of $\field \nsH_{r}$-irreducibles are very often multiplicity-free and, if not, the multiplicities are small.
\end{remark}
%see bmoduledata.txt
%<160, [ 4, 1, 1 ], [ 3, 2, 1 ], Mapping from: Free associative algebra of rank 6 over GF(1123) to Matrix Algebra of degree 160 with 6 generators over
%        GF(1123)>,
%[ <15, 2>, <9, 1>, <11, 1>, <10, 1>, <18, 1>, <12, 1>, <8, 1>, <4, 1>, <2, 1>, <6, 1>, <1, 1>, <7, 1> ],

\subsection{Combinatorics of seminormal bases}
For  $\lambda, \mu \vdash r$,  $\ell(\lambda), \ell(\mu) \leq 2$, define a bijection
\[\text{SYT}(\lambda) \times \text{SYT}(\mu) \xrightarrow{\alpha_{\lambda,\mu}} \nsbr{\text{SN}}_{\lambda,\mu} \]
inductively as follows. Maintain the notation of \eqref{e ai definition} for the outer corners of $\lambda$ and $\mu$.
In what follows let $(T, U) \in \text{SYT}(\lambda) \times \text{SYT}(\mu)$ and $i$ and  $j$ be such that $T_{a_i} = r$ and $U_{b_j} = r$.
Let $Y_\lambda$ be the tableau with entries $2c-1, 2c$ in column $c$ for each column of  $\lambda$ of height 2.
For convenience, we identify  the basis $\nsbr{\text{SN}}_{\lambda,\mu}$ with the corresponding subset of one-dimensional subspaces of $M_\lambda \tsr M_\mu$.
\begin{list}{(\roman{ctr})} {\usecounter{ctr} \setlength{\itemsep}{1pt} \setlength{\topsep}{2pt}}
\item If $\lambda \neq \mu$, set
   \[\alpha_{\lambda,\mu}(T, U) = \alpha_{\lambda-a_i,\mu-a_j}(T_{\lambda-a_i},\ U_{\lambda-a_j}).\]
\item If $\lambda = \mu$, then set
\[\alpha_{\lambda,\mu}(T, U) =
\begin{cases}
\field \nsbr{\epsilon}_+ \subseteq S^2 M_\lambda & \text{if  $(T,U) = (Y_\lambda,Y_\lambda)$}, \\
\alpha_{\lambda-a_i,\mu-a_j}(T_{\lambda-a_i},\ U_{\lambda-a_j}) & \text{otherwise},
\end{cases}
\]
where $\alpha_{\lambda-a_i,\mu-a_j}(T_{\lambda-a_i},\ U_{\lambda-a_j})$ is interpreted as a seminormal basis element of
\[\left\{
\begin{array}{ll}
M_{\lambda-a_i}\tsr M_{\lambda-a_j} \subseteq \Res_{\field \nsH_{r-1,2}} M_\lambda\tsr M_\lambda&   \text{if $i = j$,} \\
M_{\lambda-a_i}\tsr M_{\lambda-a_j} \subseteq \Res_{\field \nsH_{r-1,2}}S' \nsbr{M}_\lambda & \text{if  $i < j$},\\
M_{\lambda-a_i}\tsr M_{\lambda-a_j} \subseteq \Res_{\field \nsH_{r-1,2}}\nswedge{2}{M}_\lambda & \text{if  $i > j$}.
\end{array}
\right.
\]
\end{list}

Given Proposition \ref{p r-1 restrictions}  and Theorem \ref{t nsH irreducibles two row case}, it is clear that  $\alpha_{\lambda,\mu}$ is a well-defined bijection.

%It is important to note that $\lambda  \gdneq \mu$ implies  $\lambda - a_i \gd \mu - a_j$ in the two-row ($\ell(\lambda),\ell(\mu) \leq 2$) case.
%Taking all these modified bijections together over $i \in [k_\lambda],\ j \in [k_\mu]$ yields the desired bijection $\text{SYT}(\lambda) \times \text{SYT}(\mu) \leftrightarrow \nsbr{\text{SN}}_{\lambda,\mu}$.

%??? seems somewhat unnecessary to state a proposition here
%\begin{proposition}
%The map  $\alpha_{\lambda,\mu}$ is a bijection between $\text{SYT}(\lambda) \times \text{SYT}(\mu)$ and $\nsbr{\text{SN}}_{\lambda,\mu}$.
%\end{proposition}

%???useful could investigate the transition matrix between the two types of seminormal bases

\begin{example}
The seminormal basis element  { \setlength{\cellsize}{9pt} $\alpha_{(3,2),(3,2)}\left( \tiny\tableau{1&2&4\\3&5}, \tiny\tableau{1&3&4\\2&5} \right)$ } is a nonzero element of $S' \nsbr{M}_{(3,2)} \cap S' \nsbr{M}_{(3,1)} \cap S' \nsbr{M}_{(2,1)} \cap \nsbr{M}_{(2), (1,1)}$, where these are modules for $\field\nsH_{J_5}$, $\field\nsH_{J_{4}}$, $\field\nsH_{J_{3}}$, and $\field\nsH_{J_{2}}$, respectively.

The next two tables partially describe the bijection $\alpha_{(3,2),(3,2)}$; they give the  $\field\nsH_{J_5}$ and  $\field \nsH_{J_4}$-irreducibles that contain the seminormal basis element corresponding to each  $(T,U) \in \text{SYT}((3,2)) \times \text{SYT}((3,2))$, where row labels correspond to  $T$ and column labels correspond to  $U$; the basis element just described is in bold.
\[
\setlength{\cellsize}{10pt}\small\begin{array}{l|ccccc}
\rule[-13pt]{0pt}{0pt} & \footnotesize\tableau{1&2&3\\4&5} & \footnotesize\tableau{1&2&4\\3&5} & \footnotesize\tableau{1&3&4\\2&5} & \footnotesize\tableau{1&2&5\\3&4} & \footnotesize\tableau{1&3&5\\2&4} \\
\hline \\
\footnotesize\tableau{1&2&3\\4&5} & S' \nsbr{M}_{(3,2)} & S' \nsbr{M}_{(3,2)} & S' \nsbr{M}_{(3,2)} & S' \nsbr{M}_{(3,2)} & S' \nsbr{M}_{(3,2)} \\ \\
\footnotesize\tableau{1&2&4\\3&5} & \nswedge{2}{M}_{(3,2)} & S' \nsbr{M}_{(3,2)} & \mathbf{S' \nsbr{M}_{(3,2)}} & S' \nsbr{M}_{(3,2)} & S' \nsbr{M}_{(3,2)} \\ \\
\footnotesize\tableau{1&3&4\\2&5} & \nswedge{2}{M}_{(3,2)} & \nswedge{2}{M}_{(3,2)} & S' \nsbr{M}_{(3,2)} & S' \nsbr{M}_{(3,2)} & S' \nsbr{M}_{(3,2)} \\ \\
\footnotesize\tableau{1&2&5\\3&4} & \nswedge{2}{M}_{(3,2)} & \nswedge{2}{M}_{(3,2)} & \nswedge{2}{M}_{(3,2)} & S' \nsbr{M}_{(3,2)} & S' \nsbr{M}_{(3,2)} \\ \\
\footnotesize\tableau{1&3&5\\2&4} & \nswedge{2}{M}_{(3,2)} & \nswedge{2}{M}_{(3,2)} & \nswedge{2}{M}_{(3,2)} & \nswedge{2}{M}_{(3,2)} & \field\nsbr{\epsilon}_+ \\ \\
\end{array} \]
\[
\setlength{\cellsize}{10pt} \small
\begin{array}{l|ccccc}
\rule[-13pt]{0pt}{0pt} & \footnotesize\tableau{1&2&3 \\ 4&5} & \footnotesize\tableau{1&2&4 \\ 3&5} & \footnotesize\tableau{1&3&4 \\ 2&5} & \footnotesize\tableau{1&2&5 \\ 3&4} & \footnotesize\tableau{1&3&5 \\2&4} \\
\hline \\
\footnotesize\tableau{1&2&3 \\ 4&5} & S' \nsbr{M}_{(3,1)} & S' \nsbr{M}_{(3,1)} & S' \nsbr{M}_{(3,1)} & \nsbr{M}_{(3,1),(2,2)} & \nsbr{M}_{(3,1),(2,2)} \\ \\
\footnotesize\tableau{1&2&4 \\ 3&5} & \nswedge{2}{M}_{(3,1)} & S' \nsbr{M}_{(3,1)} & \mathbf{S' \nsbr{M}_{(3,1)}} & \nsbr{M}_{(3,1),(2,2)} & \nsbr{M}_{(3,1),(2,2)} \\ \\
\footnotesize\tableau{1&3&4 \\ 2&5} & \nswedge{2}{M}_{(3,1)} & \nswedge{2}{M}_{(3,1)} & \field\nsbr{\epsilon}_+ & \nsbr{M}_{(3,1),(2,2)} & \nsbr{M}_{(3,1),(2,2)} \\ \\
\footnotesize\tableau{1&2&5 \\ 3&4} & \nsbr{M}_{(2,2),(3,1)} & \nsbr{M}_{(2,2),(3,1)} & \nsbr{M}_{(2,2),(3,1)} & S' \nsbr{M}_{(2,2)} & S' \nsbr{M}_{(2,2)} \\ \\
\footnotesize\tableau{1&3&5 \\ 2&4} & \nsbr{M}_{(2,2),(3,1)} & \nsbr{M}_{(2,2),(3,1)} & \nsbr{M}_{(2,2),(3,1)} & \nswedge{2}{M}_{(2,2)} &  \field\nsbr{\epsilon}_+ \\ \\
\end{array}
\]
\end{example}

\section{Enumerative consequence}
\label{s enumerative consequence}
Let $C_r = \frac{1}{r+1}\binom{2r}{r}$ be the $r$-th Catalan number. Theorem \ref{t nsH irreducibles two row case} has the following corollary.
\begin{corollary} \label{c enumerative 2 row case}
The algebra $\field \nsH_{r, 2}$ has dimension $\binom{C_r}{2} - \binom{r}{\halfr} + \halfr + 2$.
\end{corollary}
\begin{proof}
It is well known that $\dim_\field(\field \H_{r, 2}) = C_r$ and therefore $\dim_\field(\field S^2 \H_{r, 2}) = \binom{C_r + 1}{2}$. On the other hand, the list of irreducibles of $\field  S^2 \H_{r, 2}$ given in Proposition-Definition \ref{p S2Hr representations} and the split semisimplicity of $\field S^2 \H_{r, 2}$ imply that
\be \label{e dim S2 H}
\dim_\field(\field S^2 \H_{r,2}) = \sum_{\substack{\lambda \gdneq \mu, \\ \ell(\lambda), \ell(\mu) \leq 2}} (f_\lambda f_\mu)^2 + \sum_{\ell(\lambda) \leq 2} \left( {\textstyle \binom{f_\lambda+1}{2}^2 + \binom{f_\lambda}{2}^2 } \right),
\ee
where $f_\lambda = \dim_\field(M_\lambda) = |\text{SYT}(\lambda)|$.

The list of $\field \nsH_{r,2}$-irreducibles from Theorem \ref{t nsH irreducibles two row case} and the split semisimplicity of $\field \nsH_{r,2}$ imply that
\be \label{e dim nsH}
\dim_\field(\field \nsH_{r,2}) = \sum_{\substack{\lambda \gdneq \mu, \\ \ell(\lambda), \ell(\mu) \leq 2}} (f_\lambda f_\mu)^2 + \sum_{\ell(\lambda) \leq 2,\ \lambda \neq (r)} {\textstyle \left( \left( \binom{f_\lambda+1}{2} - 1 \right)^2 + \binom{f_\lambda}{2}^2 \right)} + 1.
\ee
Taking the difference of the right-hand sides of (\ref{e dim S2 H}) and (\ref{e dim nsH}) then yields the first of the following string of equalities.
\[
\begin{array}{lrl}
\dim_\field(\field \nsH_{r, 2}) &= & \dim_\field(\field S^2 \H_{r,2}) + \sum_{\ell(\lambda) \leq 2,\ \lambda \neq (r)} \left( -2 \binom{f_\lambda + 1}{2} + 1 \right) \\[2mm]
&=& \binom{C_r + 1}{2} - \sum_{\ell(\lambda) \leq 2} f_\lambda^2 - \sum_{\ell(\lambda) \leq 2} f_\lambda + \left( \sum_{\ell(\lambda) \leq 2} 1 \right) + 1\\[2mm]
&=& \binom{C_r + 1}{2} - C_r - \sum_{\ell(\lambda) \leq 2} f_\lambda + (\halfr + 1)+ 1\\[2mm]
&=& \binom{C_r}{2} - \binom{r}{\halfr} + \halfr + 2.
\end{array}
\]
The second equality follows from $\dim_\field(\field S^2 \H_{r, 2}) = \binom{C_r + 1}{2}$, the third equality comes from counting the dimension of the split semisimple algebra $\field \H_{r,2}$ in two ways, and the fourth from the fact that $\dim_\CC(\ind_{P}^{\S_r} \text{triv}) = \sum_{\ell(\lambda) \leq 2} f_\lambda$, where $P$ is the maximal parabolic subgroup $(\S_r)_{S \setminus s_{\halfr}}$ of $\S_r$.
\end{proof}

\section*{Acknowledgments}
I am grateful to Ketan Mulmuley for helpful conversations and to John Wood, James Courtois, and Michael Bennett for  help typing and typesetting figures.

\bibliographystyle{plain}
\bibliography{mycitations}
\end{document}